\documentclass[12pt]{amsart}
\usepackage{amsmath, amsthm, amssymb, bm, mathrsfs, graphicx, url, fullpage, xcolor, tikz, dsfont, comment}
\usepackage{setspace}
\DeclareMathOperator{\dist}{dist}

\DeclareMathOperator{\cop}{c}
\DeclareMathOperator{\attCop}{cc}
\usepackage{hyperref}

\newtheorem{theorem}{Theorem}[section]

\newtheorem{conjecture}[theorem]{Conjecture}
\newtheorem{problem}[theorem]{Problem}
\newtheorem{observation}[theorem]{Observation}
\newtheorem{lemma}[theorem]{Lemma}

\theoremstyle{definition}

\newtheorem{question}[theorem]{Question}

\def\epsilon{\varepsilon}

\title{Cops and Attacking Robbers with Cycle Constraints
}

\author{Alexander Clow}
\address{Department of Mathematics, Simon Fraser University}
\email{alexander\_clow@sfu.ca}

\author{Melissa A. Huggan}
\address{Department of Mathematics, Vancouver Island University}
\email{melissa.huggan@viu.ca}

\author{M.E. Messinger}
\address{Department of Mathematics and Computer Science, Mount Allison University}
\email{mmessinger@mta.ca}

\date{\today}

\begin{document}

\begin{abstract}
    This paper considers the Cops and Attacking Robbers game, a variant of Cops and Robbers, where the robber is empowered to attack a cop in the same way a cop can capture the robber. In a graph $G$, the number of cops required to capture a robber in the Cops and Attacking Robbers game is denoted by $\attCop(G)$. We characterise the triangle-free graphs $G$ with $\attCop(G) \leq 2$ via a natural generalisation of the cop-win characterisation by Nowakowski and Winkler \cite{nowakowski1983vertex}. We also prove that all bipartite planar graphs $G$ have $\attCop(G) \leq 4$ and show this is tight by constructing a bipartite planar graph $G$ with $\attCop(G) = 4$. Finally we construct $17$ non-isomorphic graphs $H$ of order $58$ with $\attCop(H) = 6$ and $\cop(H)=3$. This provides the first example of a graph $H$ with $\attCop(H) - \cop(H) \geq 3$ extending work by Bonato, Finbow, Gordinowicz, Haidar, Kinnersley, Mitsche, Pra\l{}at, and Stacho \cite{bonato2014robber}.  We conclude with a list of conjectures and open problems.
\end{abstract}
 \maketitle

\section{Introduction}

\emph{Cops and Robbers} is a two-player game played on a reflexive graph $G = (V,E)$. To begin the game, the cop player places $k$ cops onto vertices of the graph, then the robber player chooses a vertex to place the robber. Game play proceeds in rounds: in each round, the cop player has a turn and then the robber player has a turn.
During the cop player's turn, the cops each move to an adjacent vertex.  Similarly, on the robber player's turn, the robber moves to an adjacent vertex.  As the graph is reflexive, there is a loop at each vertex and a cop or robber may traverse a loop during their turn.  In this paper, we equate such a move with passing.  The cop player wins if there is a cop strategy in which after finitely many moves, a cop can move onto the vertex currently occupied by the robber, thereby capturing the robber.  The robber player wins if there exists a robber strategy by which the robber can evade capture indefinitely. Both players are assumed to play optimally. The least number of cops required for the cop player to win, regardless of the robber's strategy, is the \emph{cop number} of a graph, denoted $c(G)$ for a graph $G$. If $k$ cops are sufficient to capture the robber, then the cop player has a \emph{winning cop strategy} and $c(G) \leq k$. If $k$ cops are also necessary for capture, then $c(G) = k$.

Cops and Robbers was introduced by Nowakowski and Winkler \cite{nowakowski1983vertex}, and independently by Quilliot \cite{quilliot1983problemes}, with the cop number being later introduced by Aigner and Fromme \cite{Aigner1984}. Over the last 40 years Cops and Robbers has been extensively studied. In particular, the cop number of planar graphs \cite{Aigner1984} 
and graphs of large girth \cite{bradshaw2023cop,clow2023graphs,frankl1987cops} have been studied and provide motivation for our research directions. There are a number of variants of the Cops and Robbers game within the literature, some of which affect the power dynamics between the cop player and the robber player. We recommend \cite{bonato2011game} for a general reference for Cops and Robbers and \cite{bonato2022game} for a general pursuit-evasion game reference which includes more game variants. 

The focus of this paper is the variant known as \emph{Cops and Attacking Robbers}. Introduced by Bonato, Finbow, Gordinowicz, Haidar, Kinnersley, Mitsche, Pra\l{}at, and Stacho \cite{bonato2014robber}, Cops and Attacking Robbers is played in exactly the same way as Cops and Robbers, with the added mechanic that if the robber moves onto the vertex occupied by a cop $C$, then cop $C$ is removed from the game. This mechanic is called \emph{attacking}, and we say that such a robber \emph{attacks $C$}. Importantly, the robber can only attack one cop per turn, so if the robber moves to a vertex occupied by two or more cops, then only one cop is attacked and removed. We assume both players play optimally. The \emph{attacking cop number of $G$}, denoted $\attCop(G)$, is the least integer $k$ such that there is a cop strategy in the Cops and Attacking Robbers game, whereby $k$ cops can win on graph $G$, regardless the robber's strategy. An important caveat is that if the robber begins the game on the same vertex as a cop, this does not count as attacking the cop.  We suppose all graphs are finite, reflexive and connected, although we will not draw loops in figures. To gain some intuition we recommend readers verify the cop number and attacking cop number of the graphs in Figure~\ref{fig: Small Graph Examples}.

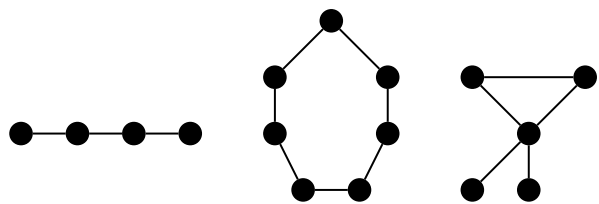
\begin{figure}[htb]
\begin{center}
    \scalebox{0.75}{
        \begin{tikzpicture}[node distance={15mm}, thick, main/.style = {draw, circle}] 
 
\node[main][fill= black] (0) at (0,1) {}; 
\node[main][fill= black] (1) at (1,1) {}; 
\node[main][fill= black] (2) at (2,1) {};
\node[main][fill= black] (3) at (3,1) {}; 

\draw [line width=1.pt] (0) -- (1) -- (2) -- (3);

\node[main][fill= black] (4) at (5,0) {}; 
\node[main][fill= black] (5) at (6,0) {}; 
\node[main][fill= black] (6) at (6.5,1) {};
\node[main][fill= black] (7) at (6.5,2) {}; 
\node[main][fill= black] (8) at (4.5,1) {}; 
\node[main][fill= black] (14) at (5.5,3) {}; 
\node[main][fill= black] (15) at (4.5,2) {}; 

\draw [line width=1.pt] (4) -- (5) -- (6) -- (7) -- (14) -- (15) -- (8) -- (4);

\node[main][fill= black] (9) at (8,0) {}; 
\node[main][fill= black] (10) at (9,0) {}; 
\node[main][fill= black] (11) at (8,2) {};
\node[main][fill= black] (12) at (10,2) {}; 
\node[main][fill= black] (13) at (9,1) {}; 

\draw [line width=1.pt] (12) -- (13) -- (11) -- (12);
\draw [line width=1.pt] (9) -- (13) -- (10);
        \end{tikzpicture}
    }
\end{center}
\caption{The graph $P_4$ (left), $C_7$ (middle), and $G$ (right). We note that $\cop(P_4)=1 = \cop(G)$ and $\cop(C_7)=2$, while $\attCop(G) = 1$, $\attCop(P_4)=2$, and $\attCop(C_7)=3$.}
\label{fig: Small Graph Examples}
\end{figure}

Notice that there are relationships between the attacking cop number and other graph parameters. First, we point out that the attacking cop number is bounded above by the domination number. If the cops begin on a dominating set, then they can capture the robber on their first turn.

\begin{observation}[\cite{bonato2014robber}]\label{Ob: Trivial upper bound}
    For all graphs $G$, we have that $\attCop(G) \leq \gamma(G)$.
\end{observation}

Next, we observe that the attacking cop number is at most twice the cop number: we simply ``double-up'' cops and follow a winning cop strategy from the classical game of Cops and Robbers.  Suppose cops $C_1$ and $C_2$ always occupy the same vertex as each other.  If the robber attacks $C_1$, then $C_2$ will capture the robber on the next cop turn.  In this case, $C_2$ acts as a ``backup" cop to $C_1$. More generally, a \emph{backup cop} is a cop, $C_j$, who stays within distance one of another cop, $C_i$. If $C_i$ is attacked during round $t$, then as $C_j$ is occupying a vertex in the closed neighbourhood of $C_i$, they capture the robber during round $t+1$. An immediate lower bound is the cop number of the graph: giving additional power to the robber only makes the situation worse for the cops. 

\begin{observation}[\cite{bonato2014robber}]\label{Ob: Trivial Cop bound}
    For all graphs $G$, we have that $\cop(G) \leq \attCop(G) \leq 2\cop(G)$.
\end{observation}

See~\cite{akhtar2024cops,bonato2014robber,das2021variations} for examples of graphs $G$ for which $cc(G) \leq c(G)+1$.  However, this inequality is not true in general.  The line graph of the Peterson graph has cop number $2$ and attacking cop number $4$~\cite{bonato2014robber}. Bonato et al.~\cite{bonato2014robber} only provided this one example where the attacking cop number and the cop number differ by more than one. Hence, there exist graphs $G$ for which $\attCop(G) - \cop(G) > 1$, and their example gives a difference of two. It remains unknown if for every integer $k \geq 3$ there exists a graph $G$, such that $\attCop(G)-\cop(G) \geq k$. We believe it to be true. In support of this conjecture, we construct  graphs $H$ for which $\attCop(H) - \cop(H) \geq 3$ in Section~\ref{Section H}. In fact, we construct $17$ graphs $H$ where $\attCop(H) = 2\cop(H) = 6$. A full list of these graphs $H$ can be found on GitHub  \cite{Clow2024Git}. We note that these $17$ graphs, $H_1,\dots, H_{17}$ are chosen from a family of $18$ candidate graphs, and we were unable to determine the value of $\attCop(H_{18})-\cop(H_{18})$, where $H_{18}$ is the final candidate graph.

Throughout the rest of the paper, we introduce new definitions as required. The paper is structured as follows. In Section~\ref{sec:tri-free_leq2}, we characterise the triangle-free graphs with attacking cop number $2$. In Section~\ref{sec:planar}, we consider planar graphs. It was shown in \cite{bonato2014robber} that outerplanar graphs have attacking cop number at most $3$, so we begin by characterising outerplanar graphs with attacking cop number $2$. Next, we provide a bipartite planar graph with attacking cop number $4$, before proving that all bipartite planar graphs have attacking cop number at most $4$. Our proof is similar to the proof that three cops can capture a robber on a planar graph in the classical game of Cops and Robbers, first proven in \cite{Aigner1984}. With respect to general planar graphs, determining an upper bound for the attacking cop number is more challenging than the cop number because we face the added complexity that shortest paths need not be $1$-guardable in Cops and Attacking Robbers \cite{bonato2014robber}, unlike in classical Cops and Robbers.
Section~\ref{Section H} is devoted to constructing graphs $H$ with $\attCop(H)-\cop(H) \geq 3$. In fact, we prove the stronger result that there exist graphs $H$ with $\cop(H) = 3$, such that $\attCop(H) = 2\cop(H) = 6$. These proofs use a combination of computer assistance and constructions which leverage the existence of certain regular graphs of large girth.  We conclude with several conjectures and open questions for future work.

\section{Triangle-Free Graphs with $\attCop(G) \leq 2$.}\label{sec:tri-free_leq2}

The goal of this section is to characterise the triangle-free graphs with attacking cop number $2$. We give our characterisation in Theorem~\ref{Thm: cc(G)=2}. Before getting to this result, we need a few definitions and lemmas. 

Let $G$ be a graph on $n$ vertices.  In~\cite{Dalhaus} Dahlhaus et al.~define vertex $v$ to {\it dominate} vertex $u$ whenever $N(u) \subseteq N[v]$ holds;   
and they further define an ordering $v_1,v_2,\dots,v_n$ of the vertices of $G$ to be a {\it domination elimination ordering} if for every $1\leq i < n$, there exists $j_i > i$ such that $v_{j_i}$ dominates $v_i$ in $G- S_{i-1}$, where $S_{i-1} = (v_1,\dots, v_{i-1})$ for $i > 1$, and $S_0 = \emptyset$. For $k \in \{1,\dots,n\}$, we define an ordering $S_k = (v_1,v_2,\dots,v_k)$ of a subset of vertices of $G$ to be a {\it k-partial domination elimination ordering} of $G$ if for every $1\leq i <k$, there exists $j_i>i$ such that $v_{j_i}$ dominates $v_i$ in $G - S_{i-1}$, where  $S_{i-1} = (v_1,\dots, v_{i-1})$ for $i > 1$, and $S_0 = \emptyset$.  Observe that an $n$-partial domination elimination ordering coincides with the definition of a domination elimination ordering.

We are now prepared to work toward proving a characterisation of triangle-free graphs with attacking cop number $2$. We begin by restating the characterisation of graphs with attacking cop number $1$.

\begin{observation}[\cite{bonato2014robber}]\label{Ob: cc(G)=1}
A graph $G$ has $\attCop(G)=1$ if and only if $\gamma(G)=1$.
\end{observation}

As a starting point, to better understand how to characterise graphs with attacking cop number 2, we consider \emph{tandem-cops}. Introduced in \cite{clarke2005tandem}, two cops must be within distance one of each other after each move with the caveat that they are allowed to move along the sides of a $4$-cycle.  The two cops therefore move in ``tandem'' and if they can capture a robber on graph $G$, then $G$ is said to be {\it tandem-win}. Certainly, if a graph $G$ is tandem-win and $\gamma(G) \neq 1$, then $\attCop(G) = 2$.
It is known from \cite{clarke2005tandem} that a triangle-free graph is tandem-win if and only if it has a domination elimination ordering.
However, there exist triangle-free graphs that have attacking cop number $2$ and are not tandem-win.  As a simple example, let $H$ be the graph obtained by subdividing each edge of $K_{1,n}$ twice and then merging the leaves to a single vertex.  As  $\gamma(H)=2$, we have $\attCop(H)=2$.

Consequently, we require more than the tandem-win result from \cite{clarke2005tandem} to characterise triangle-free graphs with attacking cop number $2$. First, we consider what happens when a triangle-free graph contains a dominated vertex. 
We note that our assumption of the graph being triangle-free is necessary for the following lemma.

\begin{lemma}\label{Lemma: retract}
    Let $G = (V,E)$ be a triangle-free graph with $\attCop(G) \geq 2$.  If $u \in V$ is a dominated vertex in $G$ and $\gamma(G-u)\geq 2$, then $\attCop(G - u) = \attCop(G)$.
\end{lemma}

\begin{proof} Let $G=(V,E)$ be a triangle-free graph with $\attCop(G) = k \geq 2$ and let $u,v \in V$ such that $N(u) \subseteq N[v]$. We let $f:V\setminus\{u\} \rightarrow V$ be the identity map $f(x)=x$ for all $x \in V\setminus \{u\}$. Note that $f$ is an injective graph homomorphism, that is if $(x,y)\in E(G-u)$, then $(f(x),f(y)) \in E$. We begin by arguing that a winning cop strategy in $G$ can be modified to form a winning cop strategy in $G-u$.

    Consider a game of Cops and Attacking Robbers played on $G-u$ with $k = \attCop(G)$ cops. Let the cops playing in $G-u$ be $C_1,\dots, C_k$, and identify each cop with the vertex they currently occupy and do the same with the robber $R$. We will mirror this game on $G$ by placing cop $C'_i$ on vertex $f(C_i)$ for all $i \in [k]$ and placing the robber $R'$ on vertex $f(R)$. As $f$ is a graph homomorphism, if a cop or robber moves in $G-u$ from $x$ to $y$, then the corresponding cop or robber in $G$ can move from $f(x)$ to $f(y)$, and as $f$ is injective and $G-u$ is an induced subgraph of $G$ as long as $x,y \in V\setminus \{u\}$, then the converse is also true. For each move that the robber $R$ makes in $G-u$ from $x$ to $y$, let the robber, $R'$, in $G$ move from $f(x)$ to $f(y)$. Then let the cops $C'_1,\dots, C'_k$ respond in $G$ using their winning strategy. 
   Notice that as $N(u) \subseteq N[v]$, the cops never need to move onto $u$ when playing optimally in $G$ except to capture the robber on $u$. We suppose without loss of generality that cop $C'_i$ never moves to $u$ except to capture the robber on $u$. Given this, if cop $C'_i$ moved from $f(x)$ to $f(y)$ let the cop $C_i$ move from $x$ to $y$ in $G-u$. This process is well-defined as $f$ is injective and for all $x \neq u$, $x \in V\setminus \{u\}$.

    As $\attCop(G) = k$ after finitely many moves the cops in $G$ will capture the robber $R'$. Given $f(R)=R'$, this implies the cops in $G-u$ have also captured the robber in $G-u$. Hence, $\attCop(G-u) \leq \attCop(G)$.

    Now suppose that $\attCop(G-u)= t$. By Observation~\ref{Ob: cc(G)=1} our assumption that $\gamma(G-u)\geq 2$ implies that $t\geq 2$. We proceed by a similar argument except now each cop $C_i$ is following their winning strategy in $G-u$ to catch the robber $R=f^{-1}(R')$ where we extend $f^{-1}$ to include that $f^{-1}(u) = v$ and allow the robber $R'$ to pursue their optimal strategy in $G$. As $\attCop(G-u)=t$, after finitely many moves the cops in $G-u$ will catch the robber $R = f^{-1}(R')$. Then the cops have captured the robber in $G$ unless $R' = u$. Suppose $R'=u$. As the cops in $G-u$ have captured the robber $R = f^{-1}(R') = f^{-1}(u) = v$, there is a cop $C_i = v$ in $G-u$, implying that there is a cop $C'_i = v$ in $G$.

    Since $G$ is triangle-free and $N(u) \subseteq N[v]$, observe that $v \in N(u)$ if and only if $\deg(u)=1$. At this point we assume the cops have just captured the robber in $G-u$ so it is the robber's turn. If $\deg(u) = 1$ we proceed as follows. On the previous turn in $G$, cop $C'_i$ is in $N[v]$ and $R' = u$. If there is another cop $C'_j$ adjacent to a vertex $z \in N[v]$, then $C'_i$ moves to $v$ and $C'_j$ moves to $z$: the cops will win during their next turn.
    Otherwise, $C'_i$ moves to a vertex in $N(v)\backslash \{u\}$ and waits for a backup cop to arrive (capturing the robber in the interim if the robber moves to $v$).  As $t \geq 2$ and $G$ is connected, after finitely many turns, there will be a cop $C'_j$ at $v$.
     At this point proceed as before. Given $v$ is the unique neighbour of $u$ during this time the robber must remain at $u$ or be immediately captured. Thus, the cops win if $\deg(u)=1$.

    If $\deg(u) \neq 1$, then $v \notin N(u)$. In this case recall that $C'_i=v$.
    Then $C'_i$ waits at $v$ until a backup cop $C'_j$ arrives at a vertex $N(u)$ after finitely many turns. 
    If the robber attacks $C'_j$ they will be captured by $C'_i$. If the robber moves elsewhere they will be captured by $C'_i$. If the robber does not move, they will be captured by $C'_j$. Thus, the cops win if $\deg(u) \neq 1$.

    Thus, if  $\gamma(G-u) \geq 2$, then $\attCop(G-u) \geq \attCop(G)$. Therefore, $\attCop(G)=\attCop(G-u)$ as required.
\end{proof}

The next step is to show that if $G$ is a triangle-free graph and with no dominated vertices, then either its attacking cop number is at least three or $\gamma(G)$ is small. That is, a triangle-free graph with large domination number and no dominated vertex must also have attacking cop number at least three. This fact will be critical to proving our characterisation.

\begin{lemma}\label{Lemma: The needed lower bound} Let $G=(V,E)$ be a triangle-free graph that contains no dominated vertices and for which $\gamma(G) \geq 3$.  Then $\attCop(G) > 2$.\end{lemma}

\begin{proof} Let $G=(V,E)$ be a triangle-free graph that contains no dominated vertices and for which $\gamma(G) \geq 3$.  Observe that the latter condition implies $\attCop(G) \geq 2$ by Observation~\ref{Ob: cc(G)=1}.  For a contradiction, assume $\attCop(G) = 2$.

Suppose the cops initially occupy $u,v \in V$ (not necessarily distinct).  Since $\gamma(G) \geq 3$, there is a vertex $z \in V$ such that $z \not\in N[u] \cup N[v]$.  The robber $R$ initially occupies such a vertex $z$ to avoid immediate capture.  As $\attCop(G)=2$, there is a winning strategy for the cops where, for some round $t>0$, the robber occupies vertex $y$ at the beginning of round $t-1$ and regardless of the move made by the robber during round $t-1$, a cop will capture the robber during round $t$. Since $R$ could pass during round $t-1$, after the cops move during round $t-1$, there must be a cop on a neighbour $x$ of $y$; this cop moves to $y$ and captures $R$ during round $t$ if the robber passes.  Thus, after the cops move during round $t-1$, a cop, say $C_1$, must occupy $x \in N(y)$.  Observe that $R$ could move to $x$ and attack $C_1$ during round $t-1$.  Since the robber is captured during round $t$,  cop $C_2$ must occupy some vertex $w \in N[x]\backslash \{y\}$. Notice that any degree $1$ vertex is dominated, so we can suppose $\deg(y)>1$.

As $\deg(y)>1$, $R$ may move to $s \in N(y)\setminus \{x\}$ during round $t-1$. Since $G$ is triangle-free, $s$ is not adjacent to $x$: in this case $C_1$ cannot capture $R$ during round $t$. However, as $R$ is captured during round $t$, $w$ must be adjacent to every vertex in $N(y) \setminus \{x\}$ to enable $C_2$ to capture $R$ during round $t$. Given $x \in N(w)$ there exists a $w \neq y$ such that $N(y) \subseteq N(w)$. Since $N(w) \subset N[w]$ this implies that $N(y) \subseteq N[w]$ and hence $y$ is a dominated vertex, which is a contradiction. Thus, $\attCop(G) > 2$ as desired.
\end{proof}

Figure~\ref{fig: cc2 on triangle} illustrates that Lemma~\ref{Lemma: The needed lower bound} does not necessarily hold when the condition that $G$ is triangle-free is removed. It is not clear if a weaker or alternative version of Lemma~\ref{Lemma: The needed lower bound} holds for graphs which contain triangles. With this, we are prepared to prove the main result of this section.
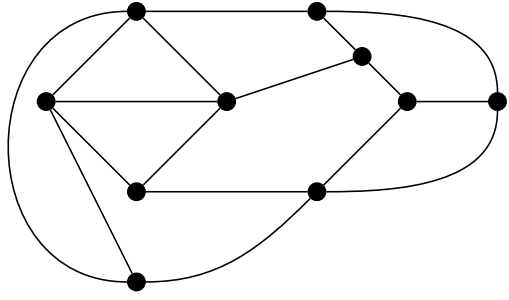
\begin{figure}[h]
\begin{minipage}[t]{0.5\textwidth}
\scalebox{0.6}{
\begin{tikzpicture}[node distance={15mm}, thick, main/.style = {draw, circle,}] 

\node[main][fill= black] (1) at (0,3) {}; 
\node[main][fill= black] (2) at (2,-1) {}; 
\node[main][fill= black] (3) at (2,1) {};
\node[main][fill= black] (4) at (2,5) {}; 
\node[main][fill= black] (5) at (4,3) {};
\node[main][fill= black] (6) at (6,1) {}; 
\node[main][fill= black] (7) at (6,5) {};
\node[main][fill= black] (8) at (7,4) {}; 
\node[main][fill= black] (9) at (8,3) {};
\node[main][fill= black] (10) at (10,3) {}; 

\draw [line width=1.pt] (1) -- (2);
\draw [line width=1.pt] (1) -- (3);
\draw [line width=1.pt] (1) -- (4);
\draw [line width=1.pt] (1) -- (5);

\draw [line width=1.pt] (2) to[out=180,in=180, looseness=1.5] (4);
\draw [line width=1.pt] (2) to[out=0,in=-135, looseness=1] (6);

\draw [line width=1.pt] (3) -- (5);
\draw [line width=1.pt] (3) -- (6);

\draw [line width=1.pt] (4) -- (5);
\draw [line width=1.pt] (4) -- (7);

\draw [line width=1.pt] (5) -- (8);

\draw [line width=1.pt] (6) -- (9);
\draw [line width=1.pt] (6) to[out=0,in=-90, looseness=1] (10);

\draw [line width=1.pt] (7) -- (8);
\draw [line width=1.pt] (7) to[out=0,in=90, looseness=1] (10);

\draw [line width=1.pt] (8) -- (9);

\draw [line width=1.pt] (9) -- (10);
\end{tikzpicture}}
\end{minipage}
\caption{An example of a graph $G$ with at least one triangle, no dominated vertices, and domination number three such that $\attCop(G)=2$.}
\label{fig: cc2 on triangle}
\end{figure}

\begin{theorem}\label{Thm: cc(G)=2}
Let $G = (V,E)$ be a triangle-free graph.  Then $G$ has $\attCop(G)\leq 2$ if and only if $\gamma(G)\leq 2$ or for some $k \in \{1,\dots,n\}$, there is a $k$-partial domination elimination ordering $S_k$ of $G$ for which $\gamma(G -S_k) \leq 2$.
\end{theorem}

\begin{proof}
Let $G = (V,E)$ be a triangle-free graph for which $\attCop(G) \leq 2$. If $\attCop(G) =1$, then Observation~\ref{Ob: cc(G)=1} implies $\gamma(G) =1$. Suppose $\attCop(G) \neq 1$ for the remainder of the proof. 

We begin by showing that if $\attCop(G) \leq 2$, then $\gamma(G)\leq 2$ or for some $k \in \{1,\dots,n\}$, there is a $k$-partial domination elimination ordering $S_k$ of $G$ for which $\gamma(G -S_k) \leq 2$. 

Suppose that $\attCop(G) = 2$. Let $G_0=G$. If $G$ contains at least one dominated vertex, then form the sequence of vertices $v_1,v_2,\dots,v_k$ inductively, by considering $G_{i-1} = G - \{v_1,\dots, v_{i-1}\}$ and letting $v_i$ be a fixed but arbitrary dominated vertex in $G_{i-1}$ whenever such a vertex exists.  For $k \geq 1$, let $v_1,\dots, v_{k}$ be any such sequence that is maximal, that is $G_k = G - \{v_1,\dots, v_k\}$ has no dominated vertices.  If $G$ contains no dominated vertex, set $k=0$. If $\gamma(G_k)\leq 2$, then we have proven our result. 

Suppose that $\gamma(G_k)\geq 3$. Since $G$ is triangle-free, $G_k$ is also triangle-free.
We may assume that $\gamma(G_i) \geq 3$ for all $1 \leq i \leq j$ because if $\gamma(G_j) \leq 2$ for some $j<k$, then $(v_1,v_2,\dots,v_j)$ would be the desired $j$-partial domination elimination ordering.
Then Lemma~\ref{Lemma: retract} implies $\attCop(G) = \attCop(G_1)= \dots = \attCop(G_k) = 2$. Hence, $G_k$ is a triangle-free graph that contains no dominated vertices and for which $\gamma(G_k) \geq 3$. Thus, Lemma~\ref{Lemma: The needed lower bound} implies that $\attCop(G_k) > 2$, contradicting $\attCop(G_k) = 2$. Hence, $\gamma(G_k) \leq 2$ as required. 

For the other direction, we will prove that if $\gamma(G)\leq 2$ or for some $k \in \{1,\dots,n\}$, there is a $k$-partial domination elimination ordering $S_k$ of $G$ for which $\gamma(G -S_k) \leq 2$, then $\attCop(G) \leq 2$. Suppose  that $\gamma(G)\leq 2$ or for some $k \in \{1,\dots,n\}$, there is a $k$-partial domination elimination ordering $S_k$ of $G$ for which $\gamma(G -S_k) \leq 2$.

If $\gamma(G) \leq 2$, then $\attCop(G) \leq 2$ by Observation~\ref{Ob: Trivial upper bound}.  Suppose $\gamma(G) \geq 3$. Then for some $k \in \{1,\dots,n\}$ there 
is a $k$-partial domination elimination ordering $S_k = (v_1,v_2,\dots,v_k)$ of $G$ and $\gamma(G-S_k) \leq 2$. Observation~\ref{Ob: Trivial upper bound} implies that $\attCop(G-S_k) \leq 2.$ Note that Observation~\ref{Ob: cc(G)=1} implies that if $\attCop(G-S_i) = 2$, then $\gamma(G-S_j) \geq 2$ for all $1\leq j \leq i$.  So we conclude that if $\attCop(G-S_i) = 2$ for some $i \in \{1,2,\dots,k\}$ then Lemma~\ref{Lemma: retract} implies $\attCop(G) = \attCop(G - S_1) = \dots = \attCop(G - S_i) = 2$. Hence, if  $\attCop(G-S_k) = 2$, then $\attCop(G) = 2$.

Otherwise, $\attCop(G-S_k) = 1$. This implies $\gamma(G-S_{k})=1$ by Observation~\ref{Ob: cc(G)=1}. Then either for all $i \in \{1,\dots, k\}$, $\attCop(G-S_i) = 1$ or there exists a $1 \leq j < k$ such that $j$ is the largest integer where $\attCop(G-S_j) > 1$. If for all $i \in \{1,\dots, k\}$, $\attCop(G-S_i) = 1$, then $\attCop(G-v_1)=1$. This implies $G-v_1$ has a universal vertex, which contradicts the assumption that $\gamma(G)\geq 3$. So there must be a largest integer $1\leq j < k$, such that $\attCop(G-S_j)>1$. 

Letting $j$ be such an integer we claim that $\attCop(G-S_j) = 2$. As $j < k$, there exists an integer $j+1 \leq k$ and $\attCop(G-S_{j+1}) = 1$. Then Observation~\ref{Ob: cc(G)=1} tells us that $\gamma(G-S_{j+1})=1$. So there exists a universal vertex $u$ in $G-S_{j+1}$. Hence, $\{u,v_{j+1}\}$ is a dominating set in $G-S_j$, which implies that $\gamma(G-S_j) \leq 2$, implying the desired result that $1 < \attCop(G-S_{j}) \leq 2$ by Observation~\ref{Ob: Trivial upper bound} and Observation~\ref{Ob: cc(G)=1}. So $\attCop(G-S_{j}) = 2$. From here Lemma~\ref{Lemma: retract} implies that $\attCop(G) = \attCop(G-S_{j}) = 2$. This proves the desired result. 

Therefore, $G$ has $\attCop(G)\leq 2$ if and only if $\gamma(G)\leq 2$ or for some $k \in \{1,\dots,n\}$, there is a $k$-partial domination elimination ordering $S_k$ of $G$ for which $\gamma(G -S_k) \leq 2$. This completes the proof.
\end{proof}

\section{Planar Graphs}\label{sec:planar}

Bonato et al.~\cite{bonato2014robber} showed that $\attCop(G) \leq c(G) + 2$ for all bipartite graphs $G$. Hence, it is known that every bipartite planar graph has attacking cop number at most $5$. Restricting our attention to bipartite planar graphs allows us to generalise both the cops' and robber's strategies for planar graphs from classical Cops and Robbers more easily than we can for general planar graphs. In this section we show every bipartite planar graph has attacking cop number at most $4$, which is tight (see Theorem~\ref{Thm: bipartite planar}). 

Before doing this however, we characterise outerplanar graphs with attacking cop number $2$. Recall that it was shown by Bonato et al.~\cite{bonato2014robber} that all outerplanar graphs have attacking cop number at most $3$. 

\begin{theorem}\label{Thm: Outerplanar}
    Let $G = (V,E)$ be an outerplanar graph with no universal vertex, and with a fixed outerplanar embedding $\Pi$. Then $\attCop(G) = 2$ if and only if there exists at most one internal $k$-face of $\Pi$ where $k\in \{5,6\}$, and no internal $t$-face of $\Pi$ where $t > 6$.
\end{theorem}

\begin{proof}
For the reverse implication, since $G$ does not have a universal vertex, $\attCop(G)\geq 2$. 
Let $G = (V,E)$ be an outerplanar graph with no universal vertex, and at most one internal $k$-face of $\Pi$ where $k\in \{5,6\}$, and no internal $t$-face of $\Pi$ where $t > 6$.
We show that given the structure of $G$, two cops suffice to capture the robber.

If the cops can place themselves on a dominating set, then the robber will be captured on the first round. If this is not the case, cops $C_1$ and $C_2$ start the game by dominating the $k$-face, where $k \in \{5,6\}$, if such a face exists. The robber can then place themselves at least distance two away from cops, and  survive the first round. This initial placement of the cops ensures that the robber cannot enter the $k$-face, and the cops can move in such a way that the $k$-face is always part of the cop territory, as $G$ is outerplanar. 

From the initial placement, on their turn, at least one cop must move to decrease their distance to the robber, while the other cop acts as a backup cop. As $G$ is outerplanar and $G$ without the vertices of the $k$-face is comprised only of three cycles, four cycles and/or tree-like structures, the cops can move in tandem along the facial walk of the infinite face in the outerplanar embedding, while remain backup to each other and preventing the robber from entering the cop territory. Here the cop territory is all vertices $u$ such that every path from the robber to $u$ include a vertex from the closed neighbourhood of one or both cops. The cops clear the faces, one at a time, as they decrease their distance to the robber. If they are on two vertices of a $3$-cycle, the cop nearer the robber remains fixed, while the other cop moves towards the robber by moving to the third vertex of the $3$-cycle. 
If the cops on a $3$-cycle are at the same distance from the robber, then they move in tandem towards the robber.
If the cops are on vertices of a $4$-cycle, they move in tandem along the $4$-cycle toward the robber. 
If the cops must traverse a portion of the graph involving a cut-vertex, the cop occupying the cut-vertex remains fixed, while the backup cop joins them on their vertex, then they move in tandem again toward the robber. 

If no such $k$-face exists, the cops initially place themselves on adjacent vertices and move in tandem toward the robber using the same strategy as before.
At each step, the cops increase their territory and get closer to the robber. As the graph is finite and the cops are moving in tandem, they are protecting against attacks and will capture the robber.

For the forward implication, suppose $G$ is an outerplanar graph with a fixed outerplanar embedding $\Pi$, no universal vertex, and suppose that $\attCop(G)=2$. Suppose by way of contradiction at least one of the following hold:
\begin{enumerate}
\item $\Pi$ contains at least two internal $k$-faces where $k \in \left\{5,6\right\}$, or
\item there exists an internal $t$-face, where $t \geq 7$.
\end{enumerate}

In both cases, we will show that two cops cannot capture the robber if the robber plays optimally. 

\begin{enumerate}
    \item Suppose $\attCop(G)=2$ and $\Pi$ contains at least two internal $k$-faces where $k \in \left\{5,6\right\}$. As $G$ is outerplanar, any
    two internal $k$-faces can share at most one edge. Hence, each face is an induced cycle and if $u$ is a vertex with neighbours on a face $f$, then $u$ has at most two neighbours on $f$, both of whom must be adjacent.
    This implies that the cop player cannot dominate both $k$-faces at the same time. The robber begins on the $k$-cycle which is not fully dominated, at distance at least two from both cops. If a cop enters the robber's neighbourhood without backup, then the robber attacks the cop. As the domination number of the graph is greater than $1$, the remaining cop cannot then catch the robber. So if the cops want to force the robber to move, they must provide each other with backup. But the robber is on a cycle of length $5$ to more, so if the cops move in tandem around the face to force the robber to move, the robber can move to remain outside of their neighbourhood without leaving the cycle. This returns the game to the same state we began. This contradicts $\attCop(G) =2$. 
    \item Suppose $\attCop(G)=2$ and there exists an internal $t$-face, where $t \geq 7$. The cop player can place the two cops anywhere on $G$. The robber player places themselves on the internal $t$-face where $t \geq 7$, at distance at least two from both cops. Similarly to the previous case, the cop player cannot dominate this face, and as an induced subgraph, this face corresponds to a cycle of length at least $7$ which has $\attCop(G) = 3$. Since $G$ is outerplanar there does not exist any vertex with more than $2$ neighbours on the $t$-face, hence the robber can remain on this face and evade capture indefinitely. 
\end{enumerate}
This completes the proof.
\end{proof}

To begin studying the attacking cop number of bipartite planar graph, we recall a useful lower bound for attacking cop number from \cite{bonato2014robber}. This result has been modified below to include the condition $\gamma(G) > \delta(G)$ as this is necessary to omit $C_5$ and $C_6$ from consideration. We note that this is itself a generalisation of a bound for cop number from \cite{Aigner1984}.

\begin{theorem}[\cite{bonato2014robber} Theorem~3]\label{Lemma: Girth >= 5 Lower Bound}
    If $G$ is a graph with girth at least $5$ and $\gamma(G) > \delta(G)$, then $\attCop(G) \geq \delta(G)+1$.   
\end{theorem}

Observe that there exists cubic planar graphs with girth $5$ and domination number greater than $4$. For an example consider the dodecahedral graph. Hence, Theorem~\ref{Lemma: Girth >= 5 Lower Bound} implies that there exists planar graphs with attacking cop number at least $4$. The following lemma extends this result  and is similar to the proof that there exists bipartite planar graphs with surrounding cop number $4$ from \cite{bradshaw2019surrounding}.

\begin{lemma}\label{Lemma: Subdivision Lower Bound} If $G$ is a graph with girth at least $5$ and $\gamma(G) > \delta(G)$, then $\attCop(H) \geq \delta(G)+1$ where $H$ is obtained from $G$ by subdividing every edge of $G$ exactly once. 
\end{lemma} 

\begin{proof} Label the vertices of $G$ as $u_1,u_2,\dots,u_n$ where $n=|V(G)|$.  Let $v_{i,j}$ be the vertex in $H$, created by subdividing edge $(u_i,u_j)$ of $G$. Note that the order of paired indices is unimportant: $v_{i,j}=v_{j,i}$. Label the remaining vertices of $H$ as $v_1,v_2,\dots,v_n$ where $v_i$ in $H$ corresponds to $u_i$ in $G$ for each $i$. We note that as the girth of $G$ is at least $5$, the girth of $H$ is at least $10$.

Assume $\attCop(H)=k \leq \delta(G)$; otherwise, $\attCop(H) \geq \delta(G)+1$ and the result holds.  We further assume $k < \gamma(G)$; otherwise $\attCop(H) \geq \gamma(G) > \delta(G)$ by Theorem~\ref{Lemma: Girth >= 5 Lower Bound} and the result holds. An implication of $\attCop(H)=k \leq \delta(G)$ is that $cc(H)<cc(G)$ by Theorem~\ref{Lemma: Girth >= 5 Lower Bound}.

Let $S = \{v_1^*,v_2^*,\dots,v_k^*\}$ be the set of vertices initially occupied by the cops in $H$ where $v_i^*$ corresponds to either a vertex in $G$ or to a subdivided edge of $G$. We will show that on $H$, the robber can avoid capture indefinitely, which will contradict the assumption that $\attCop(H) = k<\attCop(G)$.  
We next explain why the robber can avoid initially occupying a vertex in $H$ that corresponds to a subdivided edge of $G$, and also avoid being adjacent to a cop.  For each vertex $v_j^*$ in $S$, there exists $v_j \in V(H)$ such that $v_j^* \in N_H[v_j]$. Since $k<\gamma(G)$, there is a vertex $u_i \in V(G)$ adjacent to no vertex in $\{u_1,u_2,\dots,u_k\}$ in $G$.  In $H$, the robber initially occupies $v_i$. Note that $v_i$ is not adjacent to any vertex in $\{v_1^*,v_2^*,\dots,v_k^*\}$ in $H$, given $u_i$ is not adjacent to any vertex $\{u_1,u_2,\dots,u_k\}$ in $G$, and so the robber is not adjacent to a cop.

We refer to vertices $v_1,v_2,\dots,v_n$ in $H$ as {\it core} vertices. Observe that the robber initially occupies a core vertex of $H$.  We will show that there is a robber strategy for which the following property holds throughout the game: \smallskip

($\star$) once a cop moves to be adjacent to the robber, the robber can, after two turns, occupy a core vertex that is not adjacent to any cop.\smallskip

\noindent
This implies the robber wins.

As the robber occupies $v_i$, let $\deg_H(v_i)=d$ and $N_H(v_i) = \{v_{i,{m_1}},v_{i,{m_2}},\dots,v_{i,{m_d}}\}$ and suppose the robber remains on vertex $v_i$ until a cop enters $N_H(v_i)$.  If for all $t$, no cop enters $N_H(v_i)$ after $t$ cop rounds, then the robber is never captured and the robber wins. Consequently, suppose there exists a $t$ such that a cop will enter $N_H(v_i)$ after $t$ cop rounds where $t$ is the least integer satisfying this property. Without loss of generality, during round $t$ a cop $C$ occupies vertex $v_{i,{m_1}}$. 

After the cops' turn during round $t$, if there is no cop other than $C$ within distance two of $v_{m_1}$, then the robber moves to $v_{i,m_1}$, attacks the cop $C$ during round $t$, and moves to $v_{m_1}$ during round $t+1$. By assumption, no cop can be adjacent to core vertex $v_{m_1}$ at the end of round $t+1$ and ($\star$) holds. 

Suppose that after the cops have moved during round $t$, there is a cop $C'$ distinct from $C$ within distance two of $v_{m_1}$. Since $H$ has girth at least $10$, $C'$ cannot be within distance two of any of $v_{m_2},v_{m_3},\dots,v_{m_d}$. Moreover, for all $r\neq s$, $\dist(v_{m_r},v_{m_s}) \geq 6$ in $H - v_i$. Hence, $v_i$ is the unique vertex of $H$ which is within distance at most two from both $v_{m_r}$ and $v_{m_s}$ for any $r\neq s$. So any cop within distance two of $v_{m_r}$ is not within distance two of $v_{m_s}$ for any $r\neq s$, given the robber is on vertex $v_i$.
Recall that $\attCop(H) \leq \delta(G)$ and that there are two cops within distance two of $v_{m_1}$.  
This implies $k \leq d$ and so there must be a vertex $v_{m_\ell}$ (for some $\ell \in \{2,3,\dots,d\}$) such that there is no cop within distance two of $v_{m_\ell}$ after the cops move during round $t$. In this case, the robber moves to $v_{i,{m_\ell}}$ during round $t$, before moving to $v_{m_\ell}$ during round $t+1$, and there is no cop in $N[v_{m_\ell}]$ at the end of round $t+1$. In this case, ($\star$) holds.  

Repeating the above arguments provides a robber strategy which ensures ($\star$) always holds, which contradicts $\attCop(H)<\attCop(G)$ and therefore also contradicts $\attCop(H) \leq \delta(G)$. 
\end{proof}

Observe that there are planar graphs with minimum degree $3$ and girth at least $5$; for example the dodecahedral graph. Thus, Lemma~\ref{Lemma: Subdivision Lower Bound} implies that the graph formed by subdividing every edge exactly once of a minimum degree $3$ and girth $5$ planar graph has attacking cop number at least $4$. Also, notice that any graph $H$ obtained by subdividing every edge of a graph $G$ is bipartite. See Figure~\ref{fig: Dodecahedron Sub} for a picture of the dodecahedral graph with each edge subdivided.

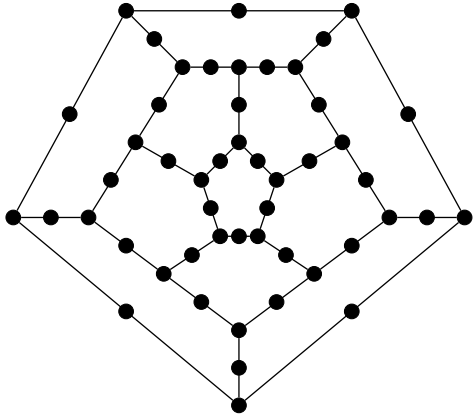
\begin{figure}[htb]
\begin{center}
    \scalebox{0.5}{
        \begin{tikzpicture}[node distance={15mm}, thick, main/.style = {draw, circle,}] 

\node[main][fill= black] (0) at (0.25,4) {};
    \node[main][fill= black] (0,1) at (-0.405,3) {};
\node[main][fill= black] (1) at (-1,2) {};
    \node[main][fill= black] (1,2) at (0,1.25) {};
\node[main][fill= black] (2) at (1,0.5) {}; 
    \node[main][fill= black] (2,3) at (2,-0.25) {};
\node[main][fill= black] (3) at (3,-1) {};
    \node[main][fill= black] (3,4) at (4,-0.25) {};
\node[main][fill= black] (4) at (5,0.5) {};
    \node[main][fill= black] (4,5) at (6,1.25) {};
\node[main][fill= black] (5) at (7,2) {};
    \node[main][fill= black] (5,6) at (6.375,3) {};
\node[main][fill= black] (6) at (5.75,4) {};
    \node[main][fill= black] (6,7) at (5.125,5) {};
\node[main][fill= black] (7) at (4.5,6) {};
    \node[main][fill= black] (7,8) at (3.75,6) {};
\node[main][fill= black] (8) at (3,6) {}; 
    \node[main][fill= black] (8,9) at (2.25,6) {};
\node[main][fill= black] (9) at (1.5,6) {};
    \node[main][fill= black] (9,0) at (0.875,5) {};

\node[main][fill= black] (10) at (2,3) {}; 
    \node[main][fill= black] (10,11) at (2.25,2.25) {};
    \node[main][fill= black] (10,0) at (1.125,3.5) {};
\node[main][fill= black] (11) at (2.5,1.5) {};
    \node[main][fill= black] (11,12) at (3,1.5) {};
    \node[main][fill= black] (11,2) at (1.75,1) {};
\node[main][fill= black] (12) at (3.5,1.5) {};
    \node[main][fill= black] (12,13) at (3.75,2.25) {};
    \node[main][fill= black] (12,4) at (4.25,1) {};
\node[main][fill= black] (13) at (4,3) {}; 
    \node[main][fill= black] (13,14) at (3.5,3.5) {};
    \node[main][fill= black] (13,6) at (4.875,3.5) {};
\node[main][fill= black] (14) at (3,4) {}; 
    \node[main][fill= black] (14,10) at (2.5,3.5) {};
    \node[main][fill= black] (14,8) at (3,5) {};

\node[main][fill= black] (15) at (-3,2) {}; 
    \node[main][fill= black] (15,16) at (0,-0.5) {};
    \node[main][fill= black] (15,1) at (-2,2) {};
\node[main][fill= black] (16) at (3,-3) {}; 
    \node[main][fill= black] (16,17) at (6,-0.5) {};
    \node[main][fill= black] (16,3) at (3,-2) {};
\node[main][fill= black] (17) at (9,2) {};
    \node[main][fill= black] (17,18) at (7.5,4.75) {};
    \node[main][fill= black] (17,5) at (8,2) {};
\node[main][fill= black] (18) at (6,7.5) {}; 
    \node[main][fill= black] (18,19) at (3,7.5) {};
    \node[main][fill= black] (18,7) at (5.25,6.75) {};
\node[main][fill= black] (19) at (0,7.5) {}; 
    \node[main][fill= black] (19,15) at (-1.5,4.75) {};
    \node[main][fill= black] (19,9) at (0.75,6.75) {};

\draw [line width=1.pt] (0) -- (1) -- (2) --  (3) -- (4) -- (5) -- (6) -- (7) -- (8) -- (9) -- (0);
\draw [line width=1.pt] (15) -- (16) -- (17) --  (18) -- (19) -- (15);
\draw [line width=1.pt] (10) -- (11) -- (12) --  (13) -- (14) -- (10);
\draw [line width=1.pt] (15) -- (1);
\draw [line width=1.pt] (16) -- (3);
\draw [line width=1.pt] (17) -- (5);
\draw [line width=1.pt] (18) -- (7);
\draw [line width=1.pt] (19) -- (9);
\draw [line width=1.pt] (2) -- (11);
\draw [line width=1.pt] (4) -- (12);
\draw [line width=1.pt] (6) -- (13);
\draw [line width=1.pt] (8) -- (14);
\draw [line width=1.pt] (0) -- (10);

\end{tikzpicture}
    }
\end{center}
\caption{The dodecahedral graph subdivided exactly once on every edge.}
\label{fig: Dodecahedron Sub}
\end{figure}

With a lower bound on the attacking cop number in hand, we proceed to obtain an upper bound. While Lemma~\ref{Lemma: Subdivision Lower Bound} generalised a robber strategy from Cops and Robbers, Lemma~\ref{Lemma: bipartite geodesic path} generalises a cops strategy from Cops and Robbers. In particular, we generalise the fact that geodesic paths are $1$-guardable, which was proven in \cite{Aigner1984}. 

Consider the game of Cops and Attacking Robbers played on a graph $G$. Let $H \subseteq G$ be a subgraph of $G$. We define $H$ to be \textit{$(k,t)$-guardable} if, in finitely many steps, $k+t$ cops
can move so that $k$ cops are placed on vertices of $H$, so that if the robber ever moves into $H$, the cops will immediately capture the robber. Note that the role of the $t$ cops is to get the $k$ cops safely positioned on $H$. Thereafter, the $t$ cops are free to move elsewhere in $G$, while the $k$ cops guard $H$. Observe that if $H$ is $(k,0)$-guardable in the classical game of Cops and Robbers, then $H$ is $(2k,0)$-guardable in the game of Cops and Attacking Robbers.  Note that there are instances where an extra $t$ cops are needed in an initial phase, until the other $k$ cops are appropriately positioned to guard a subgraph.  

Next, we will see that if $G$ is a finite bipartite graph and $P$ is a geodesic path in $G$, then $P$ is $(1,1)$-guardable. Notably, the fact that the graph $G$ is bipartite cannot be relaxed in Lemma~\ref{Lemma: bipartite geodesic path}, as there are geodesic paths which are not $(1,1)$-guardable in graphs which contain an odd cycle, as demonstrated in \cite{bonato2014robber}.

\begin{lemma}\label{Lemma: bipartite geodesic path}
If $G=(V,E)$ is a bipartite graph and $P$ is a geodesic path in $G$, then $P$ is $(1,1)$-guardable.
\end{lemma}

\begin{proof}
Let $G=(V,E)$ be a finite bipartite graph and $P=v_0,\dots, v_k$ a geodesic path in $G$. Label the cops $C$ and $C'$.  For $0 \leq i \leq k-1$, let $D_i$ denote the set of vertices distance $i$ from $v_0$ and let $D_k = \{u \in V: \dist(u,v_0) \geq k\}$. We will assume the robber never moves to a vertex adjacent to a cop; otherwise, the cop will capture the robber during the next round. 

Suppose, after some player has moved, the robber is in $D_i$ and the vertices occupied by the cop and robber are non-adjacent. We define several possible states a cop can be in;
\begin{itemize} \item if a cop occupies $v_i$, then the cop is in state ($0$); 
\item if a cop occupies $v_{i-1}$, then the cop is in state ($-1$);  
\item if a cop occupies $v_{i+1}$, then the cop is in state ($+1$). \end{itemize}

\noindent\underline{Claim 1:} After finitely many rounds, $C$ or $C'$ can achieve state ($0$) with the help of the second cop, or the cops have captured the robber.

{\emph{Proof of Claim 1.}} We assume that $C$ and $C'$ move together until they get to $v_0$ on $P$. After finitely many steps cops $C'$ and $C$ can occupy $v_0$ and $v_1$, respectively, or the cops will have captured the robber. 

Suppose the robber occupies vertex $x$. If $x \in \{v_0,v_1\}$, then the cops have captured the robber. If $x \in D_1- v_1$, then $C$ is in state ($0$) as required. Suppose $x \in D_i$ for $i \geq 2$. During each round, the cops $C$ and $C'$ will move from $v_{j-1}$ to $v_{j}$ until both the robber and $C$ occupy a vertex in $D_{j}$ for some $j \geq 2$.  Since the geodesic path is finite, this situation must occur; when it does, suppose the robber occupies $x_j$ and observe that $C$ occupies $v_j$ and $C'$ occupies $v_{j-1}$.  We consider whose turn it is to move when this situation occurs: 

\begin{enumerate} \item Suppose the robber has just moved to $x_j$.  By assumption $x_j \notin  N(v_j)$, so $C$ is in state ($0$) if $x_j \neq v_j$. If $x_j = v_j$, then the cop $C'$ will capture the robber on their next turn. 

\item Suppose the cops have just moved.  \begin{enumerate}
\item If the robber moves to a vertex $x_{j-1} \in D_{j-1}$, then by assumption, $x_{j-1} \notin N(v_{j-1})$ so $C'$ is in state ($0$), or $C$ will capture the robber if $x_{j-1} = v_{j-1}$.
\item If the robber moves to a vertex $x_j' \in D_j$, then by assumption, $x_j' \notin N(v_j)$, so $C$  is in state ($0$), or $C'$ will capture the robber if $x'_j = v_j$.
\item If the robber moves to a vertex $x_{j+1} \in D_{j+1}$, then the cops move to $v_j,v_{j+1}$ and we are at the beginning of (2) again, except the indices are increased by one.  Since the geodesic path is finite, (2)(c) can only occur finitely many times. \end{enumerate}
\end{enumerate}
This concludes the proof of Claim 1. \hfill $\diamond$ \smallskip

Hence, after finitely many moves $C$ or $C'$ can achieve state ($0$) with the help of the second cop, or the cops will capture the robber.
Once a cop has reached state ($0$), the other cop is no longer needed to guard $P$. Suppose without loss of generality $C$ reaches state ($0$). We next show how, from state ($0$), $C$ can guard $P$.

\smallskip

\noindent\underline{Claim 2:} If a cop $C$ is in state ($0$) during some round, then in all subsequent rounds, $C$ has a strategy to ensure that if the robber moves to $P$, then $C$ will immediately move to capture the robber.

{\emph{Proof of Claim 2.}} For some $i \in [k]$, suppose that $C$ occupies $v_i$ and the robber occupies $x \in D_i$ where $v_i \notin N[x]$. We will provide a strategy for $C$ to guard the path $P$ for the rest of the game. 

If it is the cops' turn, then $C$ remains on its current vertex, and thus, remains in state ($0$). If it is the robber's turn, then we consider the robber's possible moves, from $x$ to $x'$, in three cases: (i) $x' \in D_i$, (ii) $i<k$ and $x' \in D_{i+1}$, or (iii) $x' \in D_{i-1}$.  Since we assumed the robber never moves to a vertex adjacent to the cop, $x'$ is not adjacent to $v_i$, and as $v_i\notin N(x)$, we conclude that $x'\neq v_i$.
 
The remainder of the proof explains how the cop $C$ should respond in each case. Hence, if the cop is in state ($0$), then the robber cannot enter $P$ on their next move without moving adjacent to the cop $C$, thereby losing. It follows that if the robber has a strategy to enter $P$ without being captured, then this strategy will force $C$ to leave state ($0$). Thus, to show $C$ guards $P$ it is necessary and sufficient for us to show that if $C$ is forced to leave state ($0$), either $C$ can continue to guard $P$ from outside state ($0$) or $C$ can return to state ($0$).

\begin{enumerate}
\item[(i)] Suppose $x' \in D_i$.  Since $x' \notin N(v_i)$ and $x' \in D_i$, the cop $C$ remains at $v_i$ and is in state ($0$).
\item[(ii)] Suppose $i<k$ and $x' \in D_{i+1}$.  If $x' \notin N(v_{i+1})$, the cop moves to $v_{i+1}$ and is in state ($0$).  So we assume the edge $(x',v_{i+1})$ exists. After the robber's turn, the cop remains on $v_{i}$ implying that the cop is now in state ($-1$). We need to show that the cop can return to state ($0$), or the cop can guard the path $P$ while remaining in state $(-1)$ indefinitely. Note that $(x',v_{i}) \notin E$, since $x'$, $v_{i}$, and $v_{i+1}$ would form an odd cycle. See Figure~\ref{fig:caseii} (a) for a visualization. If the robber does not move on their turn, the cop does not move, and has successfully prevented the robber from entering $P$. Suppose then that the robber moves to a new vertex on their turn. From this position, the robber can move to a vertex in $D_i$, $D_{i+1}$, or $D_{i+2}$; the latter only if $i+2 \leq k$. 

First, if the robber moves to a vertex in $D_i \backslash N(v_i)$ then the cop remains at $v_i$ and is in state ($0$).

Second, if the robber moves to a vertex $x_{i+1} \in D_{i+1} \backslash (N(v_i)\cup \{x'\})$, then we note $(x_{i+1},v_{i+1}) \notin E$; otherwise $\{x',x_{i+1},v_{i+1}\}$ form a $3$-cycle. Thus, the cop moves to $v_{i+1}$ and is now in state ($0$).

Third, if $i+2 \leq k$, then suppose $R$ moves to $x_{i+2} \in D_{i+2}$. Observe that $x_{i+2} \neq v_{i+2}$; otherwise $\{x',v_{i+1},v_{i+2}\}$ forms a $3$-cycle.  Notably, $(v_{i+1},x_{i+2}) \not\in E$; otherwise $\{v_{i+1},x',x_{i+2}\}$ forms a $3$-cycle. This implies that if the robber moves to $x_{i+2}$, then the cop can safely move to $v_{i+1}$ and is now again in state $(-1)$, with an increased index on the vertices of $P$. Thus, this whole case analysis restarts from a higher index.

Notice that if at this higher index the robber decreases their index on their turn, i.e move from $D_j$ to $D_{j-1}$, then by the cop $C$ in state ($-1$) can either capture the robber, or remain on their current vertex on their turn, thereby entering state ($0$). Note that as the cop is in state ($-1$), the robber cannot attack the cop. As returning to state ($0$) benefits the cop's strategy to protect the path, we suppose that the robber does not allow to the cop to return to state ($0$).  Hence, we can suppose it is the robber's turn, and that the robber currently occupies a vertex $x_{l}\in D_{i+l}$ for some $t\geq 2$, while the cop is in state ($-1$). Without loss of generality suppose that the robber has not entered $P$ before the current turn. We will show that either the robber increases their index again, or the cop is able to enter state ($0$), or the robber enters $P$ and is captured by the cop.

There are two cases to consider. First, the robber moves from $D_{i+l}$ to a vertex in $D_{i+l}$, second the robber moves from a vertex $D_{i+l}$ to a vertex in $D_{i+l+1}$. Of course for the second option to be possible, $i+l+1 \leq k$.

We being by considering the case where the robber moves from a vertex $x_{i+l}\in D_{i+l}$ to a vertex in $x'_{i+l}\in D_{i+l}$. Without loss of generality we suppose that this is the first instance of the robber not increasing the index of the set $D_j$ they occupy on their turn. By this assumption we can suppose that on the previous $t$ turns, the robber occupies a vertex $x_{i+r} \in D_{i+r}$. Notice here that $x' = x_{i+1}$.

Recall that by assumption $(x',v_{i+1})\in E$.
Further we can suppose that $x_{i+l} \neq x'_{i+l}$, because if the robber does not move, then the cop can pass on their turn without allowing the robber to make any progress towards entering the path $P$.
Then, $x_{i+1},\dots, x_{i+l}, x'_{i+l}$ is a path of length $l+1$. Hence, if $(x'_{i+l}, v_{i+l})\in E$, then 
$$
x_{i+1},\dots, x_{i+l}, x'_{i+l}, v_{i+l},  v_{i+l-1}, \dots, v_{i+1}
$$
is a cycle of length $2t+1$. But $G$ is bipartite, so $G$ contains no odd cycle, hence, $(x'_{i+l}, v_{i+l})\not\in E$. Given $(x'_{i+l}, v_{i+l})\not\in E$, the cop $C$ can safely move from $v_{i+l-1}$ to $v_{i+l}$ on their turn, thereby entering state ($0$). We conclude that if the robber does not increase their index on every turn, then the cop can enter state ($0$) thereby guarding the path $P$. Suppose then that on every turn since moving to $x'$, the robber has increased their index.

Suppose then that the robber moves from a vertex $x_{i+l} \in D_{i+l}$ to a vertex in $x_{i+l+1} \in D_{i+l+1}$.
We must prove two facts. First, $x_{i+l+1} \neq v_{i+l+1}$, and second that $x_{i+l+1} \not\in N(v_{i+l})$. We must show the first fact in order to ensure that the cop $C$ prevents the robber from entering $P$, and we must prove the second fact to ensure that the cop can move from $v_{i+l-1}$ to $v_{i+l}$ on their turn to remain in state ($-1$).
Notice that the second fact implies the first. Hence, we will prove both facts by proving the second.

Suppose that $x_{i+l+1} \in N(v_{i+l})$, then 
$$
x_{i+1},\dots, x_{i+l}, x_{i+l+1}, v_{i+l},  v_{i+l-1}, \dots, v_{i+1}
$$
is a cycle of length $2l+1$. But as $G$ is bipartite, $G$ contains no odd cycle.
Hence,  $x_{i+l+1} \not\in N(v_{i+l})$. Thus, the robber cannot enter $P$ on their turn, and if the robber increases their index, the cop can respond by returning to state ($-1$). This concludes the proof of case (ii).

\item[(iii)] Similar to (ii) with state ($-1$) being replaced with state ($+1$).
\end{enumerate}
This concludes the proof of Claim 2. \hfill $\diamond$ 

Together, Claims 1 and 2 showed that if $G$ is a finite bipartite graph and $P$ is a geodesic path in $G$, then $P$ is $(1,1)$-guardable.
\end{proof}

\begin{figure}[htbp]
\begin{center}
    \scalebox{0.75}{
        \begin{tikzpicture}[node distance={15mm}, thick, main/.style = {draw, circle}] 

\node[main][fill= black, label=below: $v_0$] (v_0) at (0,1) {}; 
\node[main][fill= black,  label=below: $v_1$] (v_1) at (2,1) {}; 
\node[main][fill= black,  label=below: $v_{i-1}$] (v_{i-1}) at (5,1) {};
\node[main][fill= black,  label=below: $v_i$] (v_i) at (7,1) {}; 
\node[main][fill= black,  label=below: $v_{i+1}$] (v_{i+1}) at (9,1) {};
\node[main][fill= black,  label=above: $x$] (x) at (7,2) {}; 
\node[main][fill= black,  label=above: $x'$] (x') at (9,2) {};

\node[main][fill= black] (v_{i+2}) at (11,1) {};
\node[main][fill= black, label=below: $v_k$] (v_k) at (15,1) {};

\draw [line width=1.pt] (v_0) -- (v_1);
\draw [line width=1.pt] (v_{i-1}) -- (v_i) -- (v_{i+1}) -- (v_{i+2}) ;

\draw [style = dashed, line width=1.pt] (v_i) -- (x);
\draw [line width=1.pt] (2,1) -- (3,1);
\draw [line width=1.pt] (4,1) -- (5,1);
\draw [line width=1.pt] (14,1) -- (15,1);
\draw [line width=1.pt] (11,1) -- (12,1);

\draw [line width=1.pt] (x) -- (x');
\draw [style = dashed, line width=1.pt] (v_i) -- (x');
    \path (v_1) -- node[auto=false]{\ldots} (v_{i-1});
    \path (v_{i+2}) -- node[auto=false]{\ldots} (v_k);

\node at (2,0) [draw,name=D,rectangle, minimum width=1cm,minimum height=3cm,anchor=south,label=below:$D_1$] {};
\node at (5,0) [draw,name=D,rectangle, minimum width=1cm,minimum height=3cm,anchor=south,label=below:$D_{i-1}$] {};
\node at (7,0) [draw,name=D,rectangle, minimum width=1cm,minimum height=3cm,anchor=south,label=below:$D_{i}$] {};
\node at (9,0) [draw,name=D,rectangle, minimum width=1cm,minimum height=3cm,anchor=south,label=below:$D_{i+1}$] {};
\node at (15,0) [draw,name=D,rectangle, minimum width=1cm,minimum height=3cm,anchor=south,label=below:$D_{k}$] {};

\node[main][fill= black, label=below: $v_0$] (u_0) at (0,-5) {}; 
\node[main][fill= black,  label=below: $v_1$] (u_1) at (2,-5) {}; 
\node[main][fill= black,  label=below: $v_{i-1}$] (u_{i-1}) at (5,-5) {};
\node[main][fill= black,  label=below: $v_i$] (u_i) at (7,-5) {}; 
\node[main][fill= black,  label=below: $v_{i+1}$] (u_{i+1}) at (9,-5) {};
\node[main][fill= black,  label=below: $v_{i+l-1}$] (u_{i+l-1}) at (13,-5) {};
\node[main][fill= black,  label=below: $v_{i+l}$] (u_{i+l}) at (15,-5) {}; 
\node[main][fill= black,  label=below: $v_{i+l+1}$] (u_{i+l+1}) at (17,-5) {};
\node[main][fill= black,  label=below: $v_{i+l-1}$] (u_{i+l-1}) at (13,-5) {};
\node[main][fill= black,  label=above: $x_{i+l}$] (x_{i+l}) at (15,-4) {}; 
\node[main][fill= black] (y) at (13,-4) {}; 
\node[main][fill= black, label = above: $x_i$] (x_i) at (7,-4) {}; 
\node[main][fill= black,  label=above: $x_{i+1}$] (x_{i+1}) at (9,-4) {};

\draw [line width=1.pt] (2,-5) -- (3,-5);
\draw [line width=1.pt] (4,-5) -- (5,-5);
\draw [line width=1.pt] (9,-4) -- (10,-4);
\draw [line width=1.pt] (12,-4) -- (13,-4);
\draw [line width=1.pt] (15,-4) -- (16,-4);
\draw [line width=1.pt] (13,-5) -- (12.5,-5);
\draw [line width=1.pt] (11,-5) -- (11.5,-5);
\draw [line width=1.pt] (17,-5) -- (17.5,-5);
\draw [line width=1.pt] (19,-5) -- (18.5,-5);

\node[main][fill= black] (u_{i+2}) at (11,-5) {};
\node[main][fill= black, label=below: $v_k$] (u_k) at (19,-5) {};

\draw [line width=1.pt] (u_0) -- (u_1);
\draw [line width=1.pt] (u_{i-1}) -- (u_i) -- (u_{i+1}) -- (u_{i+2}) ;
\draw [line width=1.pt] (u_{i+l-1}) -- (u_{i+l}) -- (u_{i+l+1}) ;
\draw [style = dashed, line width=1.pt] (u_i) -- (x_i);
\draw [line width=1.pt] (x_i) -- (x_{i+1});
\draw [line width=1.pt] (u_{i+1}) -- (x_{i+1});
\draw [line width=1.pt] (y) -- (x_{i+l});
\path (u_1) -- node[auto=false]{\ldots} (u_{i-1});
\path (u_{i+2}) -- node[auto=false]{\ldots} (u_k);
\path (u_{i+2}) -- node[auto=false]{\ldots} (u_{i+l-1});
\path (u_{i+l+1}) -- node[auto=false]{\ldots} (u_k);
\path (x_{i+1}) -- node[auto=false]{\ldots} (y);

\node at (2,-6) [draw,name=D,rectangle, minimum width=1cm,minimum height=3cm,anchor=south,label=below:$D_1$] {};
\node at (5,-6) [draw,name=D,rectangle, minimum width=1cm,minimum height=3cm,anchor=south,label=below:$D_{i-1}$] {};
\node at (7,-6) [draw,name=D,rectangle, minimum width=1cm,minimum height=3cm,anchor=south,label=below:$D_{i}$] {};
\node at (9,-6) [draw,name=D,rectangle, minimum width=1cm,minimum height=3cm,anchor=south,label=below:$D_{i+1}$] {};

\node at (13,-6) [draw,name=D,rectangle, minimum width=1cm,minimum height=3cm,anchor=south,label=below:$D_{i+l-1}$] {};
\node at (15,-6) [draw,name=D,rectangle, minimum width=1cm,minimum height=3cm,anchor=south,label=below:$D_{i+l}$] {};
\node at (17,-6) [draw,name=D,rectangle, minimum width=1cm,minimum height=3cm,anchor=south,label=below:$D_{i+l+1}$] {};
\node at (19,-6) [draw,name=D,rectangle, minimum width=1cm,minimum height=3cm,anchor=south,label=below:$D_{k}$] {};
        \end{tikzpicture}
    }
\end{center}

\caption{(top) is visualization of the start of Case (ii); (bottom) is the visualization of the later part of Case (ii).} 

\label{fig:caseii}
\end{figure}
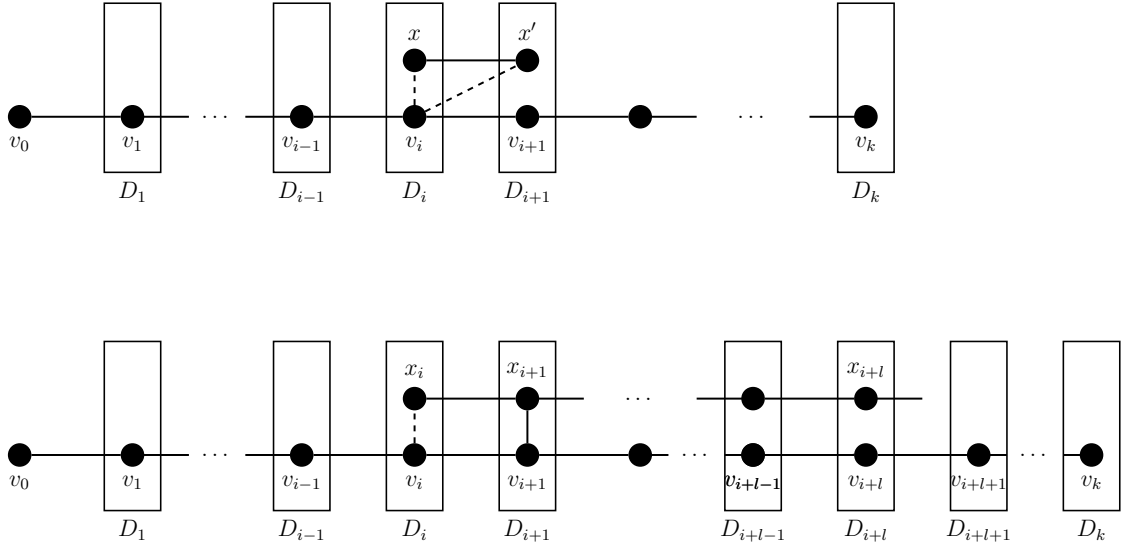

We notice that if the robber is confined to a subgraph $H$ of $G$, and $P$ is a path from $u$ to $v$ such that there is no path shorter than $P$ from $u$ to $v$ in $H$, then the same argument as in Lemma~\ref{Lemma: bipartite geodesic path} implies $P$ is $(1,1)$-guardable. Thus, paths that are not necessarily geodesic in $G$ can become $(1,1)$-guardable if these paths are geodesic in a subgraph $H$ of $G$, where the robber cannot leave $H$ without being captured, and these paths are shortest paths in $H$.

We are now prepared to prove Theorem~\ref{Thm: bipartite planar}. The proof proceeds by almost exactly the same argument as the proof that for all planar graphs $G$, $\cop(G) \leq 3$. For completeness the whole argument is included.

\begin{theorem}\label{Thm: bipartite planar}
    If $G$ is a bipartite planar graph, then $\attCop(G) \leq 4$. Furthermore, there exists a bipartite planar graph $H$ with $\attCop(H)=4$.
\end{theorem}

\begin{proof}

We begin by pointing out that the Dodecahedral graph, $G$, is $3$-regular, has girth $5$, and has domination number at least $4$. Thus, Lemma~\ref{Lemma: Subdivision Lower Bound} implies that the graph $H$ obtained by subdividing every edge of the $G$, pictured in Figure~\ref{fig: Dodecahedron Sub}, has attacking cop number at least $4$. We note that every $k$-cycle in $H$ corresponds to $k/2$-cycle in $G$. So the length of every cycle in $H$ is even, implying $H$ is bipartite. As subdividing edges does not increase a graph's genus and $G$ is planar, $H$ is planar. Thus, we have demonstrated the existence of a bipartite planar graph $H$ with $\attCop(H)\geq 4$.

We now prove that every bipartite planar graph $G$ has $\attCop(G) \leq 4$. Let $G = (V,E)$ be a fixed finite bipartite planar graph.
To prove an upper bound on the attacking cop number, we will provide a cop strategy for capturing the robber with four cops, regardless of how the robber plays. This cop strategy proceeds by partitioning the time over which the game is played into segments we call stages. The cop territory is a subgraph $H$ of $G$ consisting of all vertices $u \in V$ which the cops can prevent the robber from entering. If the robber tries to enter $H$ they will be immediately captured. In each stage $i\geq 0$, $H_i$ is a strict subgraph of $H_{i+1}$. As $G$ is finite, for some finite number $k$, $H_k=G$, implying the robber is in the cop territory, and as a result the cops will be able to capture the robber.

We prove that $H_{i}$ is a strict subgraph of $H_{i+1}$ by proving how in stage $i+1$, the cops can add a new set of vertices to their territory, while guarding $H_i$. We define the three states that the cop strategy can be in during stage $i$ and then proceed to prove that from all of these states, the cops can expand their territory in stage $i+1$ by reaching one of these three states again, with at least one new vertex in the subgraph $H_{i+1}$. 
Here stages are not a natural aspect of the game, rather they are a means for distinguishing which strategy the cops should be pursuing from a particular position.
As such, a stage begins or ends if and only if we define it to begin or end, respectively.
We suppose without loss of generality that the robber does not enter $H_i$ once stage $j\geq i$ has begun, given that they will be captured if they do so.  Suppose the current stage is $i$. The states we consider are the following:
\begin{enumerate}
    \item     
    A cop is guarding a path $P$ which is at least as short as any path that has the same endpoints as $P$, but whose internal vertices are not in $H_i$. Any path from the robber to the cop territory $H_i$, contains a vertex of $P$, while the rest of the cops are occupying vertices of $P$.
    \item   
    A cop $C_1$ guards a path $P_1$ and a cop $C_2$ guards a path $P_2$ where $P_1$ and $P_2$ are internally disjoint, but have the same endpoints.  Paths $P_1$ and $P_2$ are at least as short as any shortest path between their endpoints, but with internal vertices not in $H_i$. Any path from the robber to the cop territory $H_i$, contains a vertex from $P_1$ or $P_2$, while the rest of the cops are occupying vertices of $P_1$ or $P_2$.
    \item  A cop guards a cut vertex $v$ and any path from the robber to the cop territory, $H_i$, contains $v$, while the rest of the cops are occupying vertices in $H_i$ or are on $v$.
\end{enumerate}

Begin the game and stage $0$ with cops $C_1,C_2,C_3,C_4$ initialized on a fixed but arbitrary vertex. As stage $0$ does not begin in any of the states we have defined, we begin by proving that by the end of stage $0$ the cops can reach state (1).

\vspace{0.5cm}
\noindent\underline{Claim 1:} After being initialized the cops can reach state ($1$).\medskip

{\it Proof of Claim 1.} Let $u,v$ be any pair of vertices in $G$, with $\dist(u,v) = \max_{x,y \in V} \dist(x,y)$. Let $P$ be a shortest path between $u$ and $v$.  Lemma~\ref{Lemma: bipartite geodesic path} implies that $C_1$ can guard $P$ with the help of $C_2$ after some finite number of rounds. When $C_2$ is done assisting $C_1$, $C_2$ is located on $P$ as per our proof of Lemma~\ref{Lemma: bipartite geodesic path}. From here have $C_3$ and $C_4$ move together onto $P$, so that the robber cannot attack either cop without being captured. Then letting $H_0 = P$ the cops have reached state ($1$). Finish stage $0$ here and proceed to stage $1$. This completes the proof of Claim 1.\hfill $\diamond$

\vspace{0.5cm}
\noindent\underline{Claim 2:} If stage $i\geq 0$ ended with the cops in state ($1$), then the cops can extend their territory in stage $i+1$ while also achieving state ($1$), state ($2$), or state ($3$).\medskip

{\it Proof of Claim 2.} As the cops are currently in state ($1$), there is a path $P$ guarded by a cop, say $C_1$, and $C_2,C_3,C_4$ all occupy vertices in $P$. Let $A$ be the component of $G-P$ that the robber is in at the end of stage $i$. Note that as the cops can prevent the robber from crossing $P$, $H_i = G-V(A)$.

\vspace{0.5cm}
\noindent\underline{Case (a):} There exists exactly one vertex $x$ in $V(P)$ such that $N(x) \cap V(A)\neq \emptyset$. \medskip

Let $x\in V(P)$ be such a vertex. If $N(x)\cap V(A) = \{y\}$, then we note that $y$ is a vertex cut. Let $C_2$ move to $y$ with the help of $C_3$. As $y$ is a vertex cut and the robber is in $A$, the cops are now in state ($3$), and $V(H_{i+1}) = V(H_i) \cup  \{y\}$. 
At this point, end stage $i+1$, as we have proved the induction hypothesis in this situation.

 If $|N(x)\cap V(A)| \geq 2$, then let $y,z\in N(x)\cap V(A)$. Notice that as $A$ is a connected component of $G-P$ there is a path from $y$ to $z$ in $A$. Let $y,x_1,\dots, x_{k}, z$ be a shortest path from $y$ to $z$ in $A$. Then, $y,x_1,\dots, x_{k}, z, x$ is a shortest cycle in $G$ containing $y$ and $z$. As $G$ is bipartite, we note that $k$ is odd. Hence, $P_1 = x,y,\dots,x_{\lceil \frac{k}{2} \rceil}$ and $P_2 = x_{\lceil \frac{k}{2} \rceil}, \dots, z, x$ are shortest paths in $G$. Lemma~\ref{Lemma: bipartite geodesic path} implies that $C_2$ can guard $P_2$ with the help of $C_3$. Once, $P_2$ is guarded, the robber cannot enter $x$ without being captured by $C_2$, so $C_1$ is no longer needed to guard $P$, given there is no way for the robber to enter $G-A$ without first entering $x$ in $P_2$. From here Lemma~\ref{Lemma: bipartite geodesic path} implies that $C_1$ can guard $P_1$ with the help of $C_4$. As $x$ is a vertex on this cycle, $P_1 \cup P_2$, $x$ is guarded, and as $x$ is unique, all paths from vertices in $A$ to $H_i$ pass through the paths that are guarded. Hence, $H_{i+1}$ has $H_i$ as a strict subgraph and the cops are in state ($2$), so end stage $i+1$, as we have proved the claim in this situation.

\vspace{0.5cm}
\noindent\underline{Case (b):} There are at least two vertices $x,y$ in $V(P)$ such that $N(x) \cap V(A)\neq \emptyset$ and $N(y) \cap V(A)\neq \emptyset$.\medskip

Order the vertices of $V(P)$ from $v_1,\dots, v_k$ so that $P = v_1,\dots, v_k$. Let $v_i,v_j \in V(P)$ be the unique pair of vertices satisfying that  $N(v_i) \cap V(A)\neq \emptyset$ and $N(v_j) \cap V(A)\neq \emptyset$ and for all $v_r\in V(P)$ satisfying  $N(v_r) \cap V(A)\neq \emptyset$, it must be true that $i \leq r \leq j$. There are two situations to consider. First, if $N(v_i) \cap V(A) = N(v_j) \cap V(A) = \{y\}$. Second, the case where $N(v_i) \cap V(A) = N(v_j) \cap V(A) = \{y\}$ is not true.

If $N(v_i) \cap V(A) = N(v_j) \cap V(A) = \{y\}$, then $y$ is a vertex cut and we can extend the cop territory by moving to state ($3$), as in Case (a), when $N(x)\cap V(A) = \{y\}$. So we suppose that  $N(v_i) \cap V(A) = N(v_j) \cap V(A) \neq \{y\}$ is not true.

Then there exists a vertex $y \in N(v_i)\cap V(A)$ and $z \in N(v_j)\cap V(A)$ such that $y\neq z$.
As $A$ is a connected component there is a path from $y$ to $z$ in $A$. Let $P_1 = y,x_1,\dots,x_k,z$ be a shortest $y,z$ path in $A$.
Lemma~\ref{Lemma: bipartite geodesic path} implies that $C_2$ can guard $P_1$ with the help of $C_3$. Suppose that after $C_2$ is guarding $P_1$, $C_3$ remains on $P_1$. 

If, given a fixed plane embedding of $G$, the robber is on the exterior of the cycle \[y,x_1,\dots,x_k,z,v_j,v_{j-1}\dots,v_i,\] then all paths from the current position of the robber to $H_i$ or the interior of the aforementioned cycle pass through $P_1$ by the planarity of $G$. In this case move cops $C_1,C_4$ to $P_1$, which can be done as $C_1$ no longer needs to guard $P$. Then the cops are in state ($1$) and $H_{i+1}$ contains $H_i$ as a strict subgraph. So end stage $i+1$, as we have proved the claim in this situation. 

If the robber is on the interior of the cycle $y,x_1,\dots,x_k,z,v_j,v_{j-1}\dots,v_i$, in the same plane embedding of $G$, then let $P_2 = y,v_i,\dots,v_j,z$. As $P$ was geodesic one cop, say $C_1$, which was previously guarding $P$, can guard $v_i,\dots, v_j$, while $C_2$ guards $P_1$. Hence, $P_2$ can be guarded by $C_1$, even without any further assistance by another cop. It follows by the planarity of $G$, that the robber cannot leave the interior of the union of $P_1$ and $P_2$. Note $C_3$ is already on a vertex of $P_1$, while $C_4$ is on a vertex of $P$ implying that $C_4$ can move to a vertex of $P_2$ without being attacked by the robber. Hence, the cops can move to state ($2$) and so that $H_{i+1}$ contains $H_i$ as a strict subgraph. So end stage $i+1$, as we have proved the claim in this situation. 

This completes the proof of Claim 2. \hfill $\diamond$

\vspace{0.5cm}
\noindent\underline{Claim 3:} If stage $i\geq 0$ ended with the cops in state ($3$), then the cops can extend their territory in stage $i+1$ while also achieving state ($1$), state ($2$), or state ($3$).
\vspace{0.5cm}

{\it Proof of Claim 3.} As the cops are in state ($3$), a cop, say without loss of generality $C_1$ guards a vertex cut $v\in V$ and all paths from the robber to $H_i$ contain $v$. By assumption all the cops begin on vertices in $H_i$ and therefore all cops can move to $v$ without risk of the robber attacking them. Suppose, without loss of generality, that all cops begin on vertex $v$. Let $A$ denote the component on $G-v$ that contains the robber.

If $N(v) \cap V(A) = \{u\}$, then either $A$ has exactly one vertex $u$, or $u$ is also a cut vertex. If $A$ is a component with a single vertex then the robber occupies $u$ and the cops can capture the robber on their next turn. If the robber attacks the cops on $v$, then there are three other cops in $N[v]$ who will capture the robber on their next cop turn, so the cops will capture the robber, making $H_{i+1} = G$. 

Suppose $A$ contains at least two vertices. Then $u$ is a cut vertex. In this case have all four cops move to $u$. In this case $H_{i+1}$ is $H_i$ with the addition of $u$, implying that $H_{i+1}$ contains $H_i$ as a strict subgraph, and as all cops occupy a cut vertex, we are once again in state ($3$). So if $N(v) \cap V(A) = \{u\}$, then the induction hypothesis is satisfied, so end stage $i+1$, as we have proved the claim in this situation. 

Otherwise, $v$ has at least two neighbours in $A$. Let $u,w$ be distinct vertices in $N(v)\cap V(A)$. As $A$ is a connected component there exists a path  from $u$ to $w$ contained in $A$. Let $P = u,x_1,\dots,x_k,w$ be a shortest path from $u$ to $w$ in $A$.
Then $v,u,x_1,\dots,x_k,w$ is a cycle. As $G$ is bipartite, $k$ is odd. So $P_1 = v,u,x_1,\dots,x_{\lceil \frac{k}{2} \rceil}$ and $P_2 = x_{\lceil \frac{k}{2} \rceil}, x_{\lceil \frac{k}{2} \rceil - 1},\dots,w,v$ are both geodesic paths from $v$ to $x_{\lceil\frac{k}{2} \rceil}$, by our assumption that $P$ is shortest. Thus, $C_2$ can guard $P_1$ with the assistance of $C_3$ by Lemma~\ref{Lemma: bipartite geodesic path}. Once $C_2$ is guarding $P_1$, $C_3$ is on a vertex of $P_1$, then $C_3$ can move back to $v$ without risking being attacked, given $v$ is a vertex of $P_1$ and $P_1$ is guarded. From here $C_3$ can guard $P_2$ with the help of $C_4$. The cops are now in state ($2$) with $H_{i+1}$ equal to $H_i$ with the addition of the union of $P_1$ and $P_2$. Hence, the induction hypothesis is also satisfied if  $|N(v) \cap V(A)| > 1$. So end stage $i+1$, as we have proved the claim in this situation. 

This concludes the proof of Claim 3. \hfill $\diamond$

\vspace{0.5cm}
\noindent\underline{Claim 4:} If stage $i\geq 0$ ended with the cops in state ($2$), then the cops can extend their territory in stage $i+1$ while also achieving state ($1$), state ($2$), or state ($3$).\medskip

{\it Proof of Claim 4.} As the cops are currently in state ($2$) there is a path $P_1$ and a path $P_2$ with the same endpoints, each of which is guarded by a separate cop. Furthermore, all paths from the robber to $H_i$ contain vertices in $P_1$ or $P_2$. Without loss of generality, suppose $C_1$ guards $P_1$ and $C_2$ guards $P_2$ and $C_3$ and $C_4$ both occupy vertices in the union of $P_1$ and $P_2$. Let $A$ be the connected component of $G-P_1\cup P_2$ that the robber occupies.

\vspace{0.5cm}
\noindent\underline{Case (a):} There exists exactly one vertex $x$ in $V(P_1)\cup V(P_2)$ such that $N(x) \cap V(A)\neq \emptyset$. \medskip

Suppose $x$ is the unique vertex in $V(P_1)\cup V(P_2)$ such that $N(x) \cap V(A)\neq \emptyset$. Recall that all paths from $A$ to $H_i$ pass through $P_1$ or $P_2$. Then all paths from $A$ to $H_i$ pass through $x$. It follows that $x$ is a cut vertex. Hence, $C_3$ can move to $x$ without risking being attacked, given $x$ is a vertex in $P_1$ or $P_2$. Once $C_3$ occupies $x$ without being attacked on the first robber turn that $C_3$ occupies $x$, $C_3$ guards $x$ as either the robber is adjacent to $x$, in which case $C_3$ captures the robber, or the robber is not adjacent to $x$, at which point the robber cannot move adjacent to $x$ without being captured by $C_3$. As $x$ is a cut vertex, the cops are now in state ($3$).

At this point we have not shown that $H_{i+1}$ is strictly larger than $H_i$, so we have not proved the desired statement yet. So we do not end stage $i+1$ here. We have shown however that we can reconfigure the cops into state ($3$). Hence, from this point we can prove the desired statement by applying the cop strategy described in Claim 3.
It follows that in this case we have proven the induction statement. 

\vspace{0.5cm}
\noindent\underline{Case (b):} There exists exactly one vertex $x\in V(P_1)$ and exactly one vertex $y\in V(P_2)$ such that $N(x) \cap V(A)\neq \emptyset$ and $N(y) \cap V(A)\neq \emptyset$. \medskip

Suppose that $x\in V(P_1)$ and $y\in V(P_2)$ are the unique vertices in $P_1$ and $P_2$ such that $N(x) \cap V(A)\neq \emptyset$ and $N(y) \cap V(A)\neq \emptyset$. Let $P = x,u,x_1,\dots,x_k,v,y$ be a shortest path from $x$ to $y$ with at least one internal vertex all of whose internal vertices are in $A$. As $A$ is connected, $P$ exists.

Notice as $C_3$ and $C_4$ can move freely around the guarded paths $P_1$ and $P_2$, $C_3$ and $C_4$ can both move to $x$ without risking being attacked by the robber. From this point Lemma~\ref{Lemma: bipartite geodesic path} implies that $C_3$ can guard $P$ with the help of $C_4$. Notice that this is slightly complicated by the fact that $P$ might not be a shortest path from $x$ to $y$ in $G$, but instead is a shortest path from $x$ to $y$ through $A$. This is not a problem however, as the robber is confined to $A$ by our assumption that all paths from the robber to $H_i$ pass through $P_1$ or $P_2$.
Hence, $P$ is in fact $(1,1)$-guardable under these conditions, as the robber can never take advantage of a path that has vertices in $H_i$ from $u$ to $v$, which is potentially shorter than $P$.

As $x$ and $y$ are unique, $\{x,y\}$ is a vertex cut. Hence, once $C_3$ guards $P$ all paths from the robber to vertices in $V(H_i)\cup V(P)$ pass through $P$. It follows that now the cops are in state ($1$). Furthermore, $P$ contains at least one internal vertex, so $H_{i+1}$, defined by $H_i$ with the addition of all the internal vertices of $P$, contains $H_i$ as a strict subgraph. So we end stage $i+1$, as we have proved the claim in this situation.

\vspace{0.5cm}
\noindent\underline{Case (c):} Either $P_1$ or $P_2$ contain distinct vertices $x,y$ such that $N(x)\cap V(A)\neq \emptyset$ and  $N(y)\cap V(A)\neq \emptyset$.

Without loss of generality, suppose $P_1$ contains at least two distinct vertices $x,y$ with neighbours in $A$.
Suppose that $P_1 = v_1,\dots, v_k$, and let $v_i,v_j$ be the vertices in $P_1$ chosen so that $v_i$ and $v_j$ have neighbours in $A$ and for all $v_r$ if $v_r$ has a neighbour in $A$, then $i \leq r\leq j$.

Let $P_3$ be a shortest path from $v_i$ to $v_j$ with at least one internal vertex all of whose internal vertices are in $A$. Notice that $P_3$ exists as $v_i$ and $v_j$ have neighbours in $A$ and $A$ is a connected component.

Letting $P_3 = v_i,x_1,\dots, x_t,v_j$, we define the path $P_4 = v_1,\dots,v_i,x_1,\dots,x_t,v_j,\dots,v_k$. As $P_3$ is a shortest path from $v_i$ to $v_j$ whose internal vertices are in $A$, and $P_1$ is at least as short as any shortest path from $v_1$ to $v_k$ with internal vertices not in $H_i$, it follows that
$P_4$ is a shortest path from $v_1$ to $v_k$ with some vertex in $A$ as an internal vertex.
Given $P_1$ and $P_2$ are guarded by $C_1$ and $C_2$, $P_1$ and $P_2$ are in $H_i$, this implies that $P_4$ can be guarded by $C_3$ with the help of $C_4$, given Lemma~\ref{Lemma: bipartite geodesic path}.

Once $C_3$ is guarding $P_4$, the planarity of $G$ implies that the robber is either on the interior of the cycle $v_i,v_{i+1},\dots, v_j,x_t,x_{t-1},\dots, x_1$, or the robber is on the exterior of this cycle. If the robber is on the interior of this cycle, then $P_5 = v_i,v_{i+1},\dots, v_j$ is a subpath of $P_1$, hence, $C_1$ can move from guarding $P_1$ to guarding $P_5$, without ever allowing the robber to leave the interior of the cycle. At this point $P_2$ is on the exterior of the cycle $P_3\cup P_5$, so $P_2$ does not need to be guarded, as the robber cannot cross $P_4$ or $P_5$. Thus, cops $C_2$ and $C_4$ neither of which is on the interior of the cycle can make their way to vertices in $P_4$ or $P_5$ and the cops are again in state ($2$), while $H_{i+1}$ contains $H_i$ as a strict subgraph, given $P_4$ contains at least one internal vertex. Otherwise, if the robber is on the exterior of the cycle formed by $P_3 \cup P_5$, then the robber must be on the interior of the cycle formed by $P_2$ and $P_4$, by the planarity of $G$ and our choice of $v_i$ and $v_j$. By our assumption that for all $v_r$ with neighbours in $A$, $i \leq r \leq j$, all paths from the robber to $H_{i+1}$, which is now all vertices not on the interior of $P_2$ and $P_4$, contain a vertex in $P_2$ or $P_4$, implying that the cops are in state ($2$). So again the cops are again in state ($2$), while $H_{i+1}$ contains $H_i$ as a strict subgraph, given $P_4$ contains at least one internal vertex. So end stage $i+1$, as we have proved the claim in this situation.
 
This completes the proof of Claim 4.\hfill $\diamond$

As this covers every possible case in our induction, we have now shown that for all $i\geq 0$, there is a cop strategy to make $H_{i+1}$ contain $H_i$ as a strict subgraph, until such a time that $H_k = G$, for some finite $k$. Also recall that $H_k$ is the cop territory, so if the robber is on a vertex in $H_k$, then they will be captured by the cops. As $G = H_k$, the robber must be in the cop territory, hence we conclude that four cops will capture the robber. This concludes the proof.
\end{proof}

\section{Constructing Graphs where $\attCop(H)-\cop(H)=3$}\label{Section H}

One of the main unanswered questions raised by Bonato et al.~\cite{bonato2014robber} is how large the difference can be between cop number and attacking cop number? It was shown in \cite{bonato2014robber} that for any bipartite graph $G$,  $\attCop(G) - \cop(G) \leq 2$. Moreover, it was shown in \cite{bonato2014robber} that $\attCop(G) \leq \cop(G) + 2\Delta(G) - 2$, so for graphs with bounded maximum degree, the cop number and attacking cop number are at most a constant apart. Is there an integer $N$, such that for all graphs $G$, $\attCop(G) - \cop(G) \leq N$? If not, is it possible there exists a constant $0< \epsilon \leq 1$ for which there are infinitely many integers $k$ and graphs $G_k$ such that $\cop(G_k)= k$ and $\attCop(G_k) \geq (1+\epsilon)k$?

Unfortunately we were not able to answer these questions. However, we were able to make progress. As mentioned in the introduction, Bonato et al.~\cite{bonato2014robber} showed that the line graph of the Peterson graph has cop number $2$ and attacking cop number $4$. Hence, if $N$ exists, then $N \geq 2$. Furthermore, we cannot discount the possibility that there exists graphs with arbitrarily large cop number, whose attacking cop number is twice their cop number. 

We prove that if $N$ exists, then $N \geq 3$ by providing $17$ graphs with cop number $3$ and attacking cop number $6$. We note that our construction, see Lemma~\ref{Lemma: Square Lower Bound}, does not rely on hypergraphs, unlike Lemma~8 from \cite{bonato2014robber}, and as a result may be easier to use when constructing graphs with large attacking cop number. 
It may also be true that for $k\geq 4$, our method produces graphs  $H$ with $\attCop(H) = 2k$ and $\cop(H) = k$.
Unfortunately, given we use computer assistance to verify the cop number of our constructions, we cannot check any case where $k\geq 4$. This is because the smallest graphs $H$ which satisfy these assumptions for $k\geq 4$ are too large for our computers to handle.

The next lemma is the key observation used in our constructions. Before exploring that, we need the definition of the square of a graph. Given a graph $G$, we define the square of $G$, denoted $G^2$, to be the graph obtained by adding edges $(u,v)$ to $G$ for every pair $u,v\in V$ such that $\dist(u,v)=2$ in $G$. We note that $G^2$ is called the square of $G$ because the $G^2$ is the reflexive graph containing no multiedges associated with
the adjacency matrix of $G$ squared. We note that the square, and more generally the $k^{\text{th}}$ power of a graph are well-studied objects, particularly as they relate to many algebraic graph invariants, as well as invariants related to graph distance.

\begin{lemma}\label{Lemma: Square Lower Bound}
    If $G = (V,E)$ is a graph with girth at least $9$ and minimum degree $\delta \geq 3$, then $\attCop(G^2-E)\geq \min\{2\delta,\gamma(G^2-E)\}$.
\end{lemma}

\begin{proof}
    Let $G = (V,E)$ be a graph with girth at least $9$ and minimum degree $\delta \geq 3$, let $H = G^2 - E$, and $k = \min\{2\delta,\gamma(H)\}$. Thus, $(u,v)\in E(H)$ if and only if $\dist_G(u,v)=2$.  If $k = \gamma(H)$, then we note by Observation~\ref{Ob: Trivial upper bound}, that $\attCop(H) \leq k$.

    We aim to show that if $t<2\delta$ cops do not begin the game in a dominating set, then the robber will win the game. 
    Let $t<2\delta$ be a fixed integer. 

    Suppose cops $C_1,\dots, C_t$ begin the game in an ideal formation in $H$ to capture the robber $R$, subject to the constraint that this formation is not a dominating set. Identify each cop and the robber with the current vertex that they occupy, and update this as the game progresses. As the cops do not begin the game on a dominating set, there exists a vertex $z \in V$ such that $z \notin N[C_i]$ for any $1 \leq i \leq t$.
    Let $R$ begin the game on such a vertex $z$. Thus, the robber is guaranteed at least one turn to move before being captured.

    Now suppose for the sake of contradiction that after some series of play it is the robber's turn to move, but no matter where they move they will be captured on the next cop turn. If such a situation is impossible, then the robber will never be captured and it follows that $t$ cops beginning in a non-dominating set cannot catch the robber.

    Given the definition of $H$, for each $v \in N_G(R)$, the neighbourhood of $R$ in $G$, the vertices $N_G(v)$ form a clique which we call $K_v$ in $H$. Also note that as $G$ has girth at least $9$ for all $u,v \in N_G(R)$, $V(K_v) \cap V(K_u) = \{R\}$. Furthermore, as $G$ has girth at least $9$ for all $u, v \in N_G(R)$ and $x \in V(K_v)\setminus \{R\}$, $y \in V(K_u)\setminus \{R\}$, $N(x) \cap N(y) = \{R\}$ in $H$ if $u\neq v$ and $N(x) \cap N(y) = V(K_u)=V(K_v)$ in $H$ if $u = v$. 
    
    As for all $v \in N_G(R)$ the cops will capture the robber if the robber moves to $x \in V(K_v)\setminus \{R\}$, there exists a cop $C$ such that $x \in N_H[C]$. Then either there exists a cop $C \in V(K_v)$ or for all $x \in V(K_v)\setminus \{R\}$, there exists a cop $C \in N_H(x)\setminus V(K_v)$.  If there is a cop $C$ in $V(K_v)$, then there must be a second cop adjacent to $C$, otherwise the robber can avoid being captured for one more turn by attacking $C$. As for all $x \in V(K_v) \setminus \{R\}$ and $ y \in V(K_u) \setminus \{R\}$, where $u \neq v$,
    $$\big(  N(x)\setminus \{R\} \big)\cap \big( N(y)\setminus \{R\} \big) = \emptyset
    $$
    the cop $C$ in $V(K_v)$ and its neighbouring cop cannot guard vertices $y \in V(K_u)\setminus \{R\}$. Furthermore, given for all $x,y \in V(K_v) \setminus \{R\}$,
    $$
    \big( N(x)\setminus V(K_v) \big) \cap \big( N(y)\setminus V(K_v)  \big) = \emptyset
    $$
    if there is no cop in $V(K_v)$, then for each $x \in V(K_v)\setminus \{R\}$ there must exists a distinct cop $C_x \in N_H(x) \setminus V(K_v)$. Additionally for each of these cops $C_x$, $N_H(C_x) \cap N_H(R) = \{x\}$. Hence, if there is no cop in $V(K_v)$, then there must be at least $|V(K_v)\setminus \{R\}| = \deg_G(v) -1 \geq \delta -1 \geq 2$ cops who are guarding vertices in $K_v$, and these cops cannot be guarding any other neighbour of $R$ in $H$.

    Thus, for each $v \in N_G(R)$ there must be at least $2$ cops guarding vertices in $V(K_v)$ and these cops cannot guard vertices in $V(K_u)$ for $u \neq v$. But this implies $t \geq 2\deg_G(v) \geq 2\delta$ contradicting our assumption that $t < 2\delta$. Therefore, if $t<2\delta$ cops do not begin the game in a dominating set, then the robber will win the game. This implies the desired result that $\attCop(H) \geq \min\{2\delta,\gamma(G^2-E)\}$. This concludes the proof.
\end{proof}

Given Lemma~\ref{Lemma: Square Lower Bound}, to demonstrate an $\epsilon>0$ such that there are infinitely many integers $k$ and graphs $G_k$ such that $\cop(G_k)= k$ and $\attCop(G_k) \geq (1+\epsilon)k$, it suffices to find a family of graphs $G$ with arbitrarily large minimum degree $\delta$ and girth at least $9$, such that $\cop(G_k^2-E) \leq \frac{2\delta}{1+\epsilon}$. This seems to be a hard problem for two reasons. First, graphs with large girth and large minimum degree are quite unintuitive. For proof of this one need look no further than the famous problem of proving the existence of graphs with high girth and high chromatic number, which was solved by Erd\H{o}s. Second, we are aware of no way to bound the cop number of $G^2-E$ from above without computer assistance.

Despite this challenge, or perhaps because of it, we believe this problem to be quite interesting. It seems that in order to understand how the attacking cop number and cop number are different, we require a much better understanding of the cop number. Thus, the attacking cop number should be of interest not just for its own sake, but as understanding it will likely require better understanding of the cop number, and in all probability the development of new tools that can be applied in other pursuit-evasion games.

\begin{figure}
\begin{minipage}[t]{0.495\textwidth}
\scalebox{0.4}{
  \begin{tikzpicture}
      \draw
        (0.0:10) node (1){1}
        (6.207:10) node (2){2}
        (12.414:10) node (3){3}
        (18.621:10) node (4){4}
        (24.828:10) node (5){5}
        (31.034:10) node (6){6}
        (37.241:10) node (7){7}
        (43.448:10) node (8){8}
        (49.655:10) node (9){9}
        (55.862:10) node (10){10}
        (62.069:10) node (11){11}
        (68.276:10) node (12){12}
        (74.483:10) node (13){13}
        (80.69:10) node (14){14}
        (86.897:10) node (15){15}
        (93.103:10) node (16){16}
        (99.31:10) node (17){17}
        (105.517:10) node (18){18}
        (111.724:10) node (19){19}
        (117.931:10) node (20){20}
        (124.138:10) node (21){21}
        (130.345:10) node (22){22}
        (136.552:10) node (23){23}
        (142.759:10) node (24){24}
        (148.966:10) node (25){25}
        (155.172:10) node (26){26}
        (161.379:10) node (27){27}
        (167.586:10) node (28){28}
        (173.793:10) node (29){29}
        (180.0:10) node (30){30}
        (186.207:10) node (31){31}
        (192.414:10) node (32){32}
        (198.621:10) node (33){33}
        (204.828:10) node (34){34}
        (211.034:10) node (35){35}
        (217.241:10) node (36){36}
        (223.448:10) node (37){37}
        (229.655:10) node (38){38}
        (235.862:10) node (39){39}
        (242.069:10) node (40){40}
        (248.276:10) node (41){41}
        (254.483:10) node (42){42}
        (260.69:10) node (43){43}
        (266.897:10) node (44){44}
        (273.103:10) node (45){45}
        (279.31:10) node (46){46}
        (285.517:10) node (47){47}
        (291.724:10) node (48){48}
        (297.931:10) node (49){49}
        (304.138:10) node (50){50}
        (310.345:10) node (51){51}
        (316.552:10) node (52){52}
        (322.759:10) node (53){53}
        (328.966:10) node (54){54}
        (335.172:10) node (55){55}
        (341.379:10) node (56){56}
        (347.586:10) node (57){57}
        (353.793:10) node (58){58};
      \begin{scope}[-]
        \draw (1) to (2);
        \draw (1) to (58);
        \draw (1) to (9);
        \draw (2) to (3);
        \draw (2) to (27);
        \draw (3) to (4);
        \draw (3) to (42);
        \draw (4) to (5);
        \draw (4) to (13);
        \draw (5) to (6);
        \draw (5) to (47);
        \draw (6) to (7);
        \draw (6) to (55);
        \draw (7) to (8);
        \draw (7) to (34);
        \draw (8) to (9);
        \draw (8) to (20);
        \draw (9) to (10);
        \draw (10) to (11);
        \draw (10) to (39);
        \draw (11) to (12);
        \draw (11) to (52);
        \draw (12) to (13);
        \draw (12) to (31);
        \draw (13) to (14);
        \draw (14) to (15);
        \draw (14) to (22);
        \draw (15) to (16);
        \draw (15) to (36);
        \draw (16) to (17);
        \draw (16) to (54);
        \draw (17) to (18);
        \draw (17) to (41);
        \draw (18) to (19);
        \draw (18) to (48);
        \draw (19) to (20);
        \draw (19) to (29);
        \draw (20) to (21);
        \draw (21) to (22);
        \draw (21) to (44);
        \draw (22) to (23);
        \draw (23) to (24);
        \draw (23) to (57);
        \draw (24) to (25);
        \draw (24) to (40);
        \draw (25) to (26);
        \draw (25) to (33);
        \draw (26) to (27);
        \draw (26) to (53);
        \draw (27) to (28);
        \draw (28) to (29);
        \draw (28) to (37);
        \draw (29) to (30);
        \draw (30) to (31);
        \draw (30) to (56);
        \draw (31) to (32);
        \draw (32) to (33);
        \draw (32) to (45);
        \draw (33) to (34);
        \draw (34) to (35);
        \draw (35) to (36);
        \draw (35) to (50);
        \draw (36) to (37);
        \draw (37) to (38);
        \draw (38) to (39);
        \draw (38) to (46);
        \draw (39) to (40);
        \draw (40) to (41);
        \draw (41) to (42);
        \draw (42) to (43);
        \draw (43) to (44);
        \draw (43) to (51);
        \draw (44) to (45);
        \draw (45) to (46);
        \draw (46) to (47);
        \draw (47) to (48);
        \draw (48) to (49);
        \draw (49) to (50);
        \draw (49) to (58);
        \draw (50) to (51);
        \draw (51) to (52);
        \draw (52) to (53);
        \draw (53) to (54);
        \draw (54) to (55);
        \draw (55) to (56);
        \draw (56) to (57);
        \draw (57) to (58);
      \end{scope}
    \end{tikzpicture}
}
\end{minipage}
\begin{minipage}[t]{0.495\textwidth}
\scalebox{0.4}{
\begin{tikzpicture}
      \draw
        (0.0:10) node (1){1}
        (6.207:10) node (2){2}
        (12.414:10) node (3){3}
        (18.621:10) node (4){4}
        (24.828:10) node (5){5}
        (31.034:10) node (6){6}
        (37.241:10) node (7){7}
        (43.448:10) node (8){8}
        (49.655:10) node (9){9}
        (55.862:10) node (10){10}
        (62.069:10) node (11){11}
        (68.276:10) node (12){12}
        (74.483:10) node (13){13}
        (80.69:10) node (14){14}
        (86.897:10) node (15){15}
        (93.103:10) node (16){16}
        (99.31:10) node (17){17}
        (105.517:10) node (18){18}
        (111.724:10) node (19){19}
        (117.931:10) node (20){20}
        (124.138:10) node (21){21}
        (130.345:10) node (22){22}
        (136.552:10) node (23){23}
        (142.759:10) node (24){24}
        (148.966:10) node (25){25}
        (155.172:10) node (26){26}
        (161.379:10) node (27){27}
        (167.586:10) node (28){28}
        (173.793:10) node (29){29}
        (180.0:10) node (30){30}
        (186.207:10) node (31){31}
        (192.414:10) node (32){32}
        (198.621:10) node (33){33}
        (204.828:10) node (34){34}
        (211.034:10) node (35){35}
        (217.241:10) node (36){36}
        (223.448:10) node (37){37}
        (229.655:10) node (38){38}
        (235.862:10) node (39){39}
        (242.069:10) node (40){40}
        (248.276:10) node (41){41}
        (254.483:10) node (42){42}
        (260.69:10) node (43){43}
        (266.897:10) node (44){44}
        (273.103:10) node (45){45}
        (279.31:10) node (46){46}
        (285.517:10) node (47){47}
        (291.724:10) node (48){48}
        (297.931:10) node (49){49}
        (304.138:10) node (50){50}
        (310.345:10) node (51){51}
        (316.552:10) node (52){52}
        (322.759:10) node (53){53}
        (328.966:10) node (54){54}
        (335.172:10) node (55){55}
        (341.379:10) node (56){56}
        (347.586:10) node (57){57}
        (353.793:10) node (58){58};
      \begin{scope}[-]
        \draw (1) to (3);
        \draw (1) to (8);
        \draw (1) to (10);
        \draw (1) to (27);
        \draw (1) to (49);
        \draw (1) to (57);
        \draw (3) to (5);
        \draw (3) to (13);
        \draw (3) to (27);
        \draw (3) to (41);
        \draw (3) to (43);
        \draw (8) to (6);
        \draw (8) to (10);
        \draw (8) to (19);
        \draw (8) to (21);
        \draw (8) to (34);
        \draw (10) to (12);
        \draw (10) to (38);
        \draw (10) to (40);
        \draw (10) to (52);
        \draw (27) to (25);
        \draw (27) to (29);
        \draw (27) to (37);
        \draw (27) to (53);
        \draw (49) to (18);
        \draw (49) to (35);
        \draw (49) to (47);
        \draw (49) to (51);
        \draw (49) to (57);
        \draw (57) to (22);
        \draw (57) to (24);
        \draw (57) to (30);
        \draw (57) to (55);
        \draw (2) to (4);
        \draw (2) to (9);
        \draw (2) to (26);
        \draw (2) to (28);
        \draw (2) to (42);
        \draw (2) to (58);
        \draw (4) to (6);
        \draw (4) to (12);
        \draw (4) to (14);
        \draw (4) to (42);
        \draw (4) to (47);
        \draw (9) to (7);
        \draw (9) to (11);
        \draw (9) to (20);
        \draw (9) to (39);
        \draw (9) to (58);
        \draw (26) to (24);
        \draw (26) to (28);
        \draw (26) to (33);
        \draw (26) to (52);
        \draw (26) to (54);
        \draw (28) to (19);
        \draw (28) to (30);
        \draw (28) to (36);
        \draw (28) to (38);
        \draw (42) to (17);
        \draw (42) to (40);
        \draw (42) to (44);
        \draw (42) to (51);
        \draw (58) to (23);
        \draw (58) to (48);
        \draw (58) to (50);
        \draw (58) to (56);
        \draw (5) to (7);
        \draw (5) to (13);
        \draw (5) to (46);
        \draw (5) to (48);
        \draw (5) to (55);
        \draw (13) to (11);
        \draw (13) to (15);
        \draw (13) to (22);
        \draw (13) to (31);
        \draw (41) to (16);
        \draw (41) to (18);
        \draw (41) to (24);
        \draw (41) to (39);
        \draw (41) to (43);
        \draw (43) to (21);
        \draw (43) to (45);
        \draw (43) to (50);
        \draw (43) to (52);
        \draw (6) to (34);
        \draw (6) to (47);
        \draw (6) to (54);
        \draw (6) to (56);
        \draw (12) to (14);
        \draw (12) to (30);
        \draw (12) to (32);
        \draw (12) to (52);
        \draw (14) to (16);
        \draw (14) to (21);
        \draw (14) to (23);
        \draw (14) to (36);
        \draw (47) to (18);
        \draw (47) to (38);
        \draw (47) to (45);
        \draw (7) to (20);
        \draw (7) to (33);
        \draw (7) to (35);
        \draw (7) to (55);
        \draw (46) to (32);
        \draw (46) to (37);
        \draw (46) to (39);
        \draw (46) to (44);
        \draw (46) to (48);
        \draw (48) to (17);
        \draw (48) to (19);
        \draw (48) to (50);
        \draw (55) to (16);
        \draw (55) to (30);
        \draw (55) to (53);
        \draw (34) to (25);
        \draw (34) to (32);
        \draw (34) to (36);
        \draw (34) to (50);
        \draw (54) to (15);
        \draw (54) to (17);
        \draw (54) to (52);
        \draw (54) to (56);
        \draw (56) to (23);
        \draw (56) to (29);
        \draw (56) to (31);
        \draw (20) to (18);
        \draw (20) to (22);
        \draw (20) to (29);
        \draw (20) to (44);
        \draw (33) to (24);
        \draw (33) to (31);
        \draw (33) to (35);
        \draw (33) to (45);
        \draw (35) to (15);
        \draw (35) to (37);
        \draw (35) to (51);
        \draw (19) to (17);
        \draw (19) to (21);
        \draw (19) to (30);
        \draw (21) to (23);
        \draw (21) to (45);
        \draw (11) to (31);
        \draw (11) to (39);
        \draw (11) to (51);
        \draw (11) to (53);
        \draw (39) to (24);
        \draw (39) to (37);
        \draw (38) to (36);
        \draw (38) to (40);
        \draw (38) to (45);
        \draw (40) to (17);
        \draw (40) to (23);
        \draw (40) to (25);
        \draw (52) to (50);
        \draw (31) to (29);
        \draw (31) to (45);
        \draw (51) to (44);
        \draw (51) to (53);
        \draw (53) to (16);
        \draw (53) to (25);
        \draw (30) to (32);
        \draw (32) to (25);
        \draw (32) to (44);
        \draw (15) to (17);
        \draw (15) to (22);
        \draw (15) to (37);
        \draw (22) to (24);
        \draw (22) to (44);
        \draw (16) to (18);
        \draw (16) to (36);
        \draw (23) to (25);
        \draw (36) to (50);
        \draw (37) to (29);
        \draw (18) to (29);
      \end{scope}
    \end{tikzpicture}

}
\end{minipage}
    \caption{The graph $G_1$ (left) from \cite{biggs1980trivalent,brinkmann1995smallest} and the graph $H_1 = G_1^2 - E(G_1)$ (right).}
    \label{fig: G_1,H_1}
\end{figure}
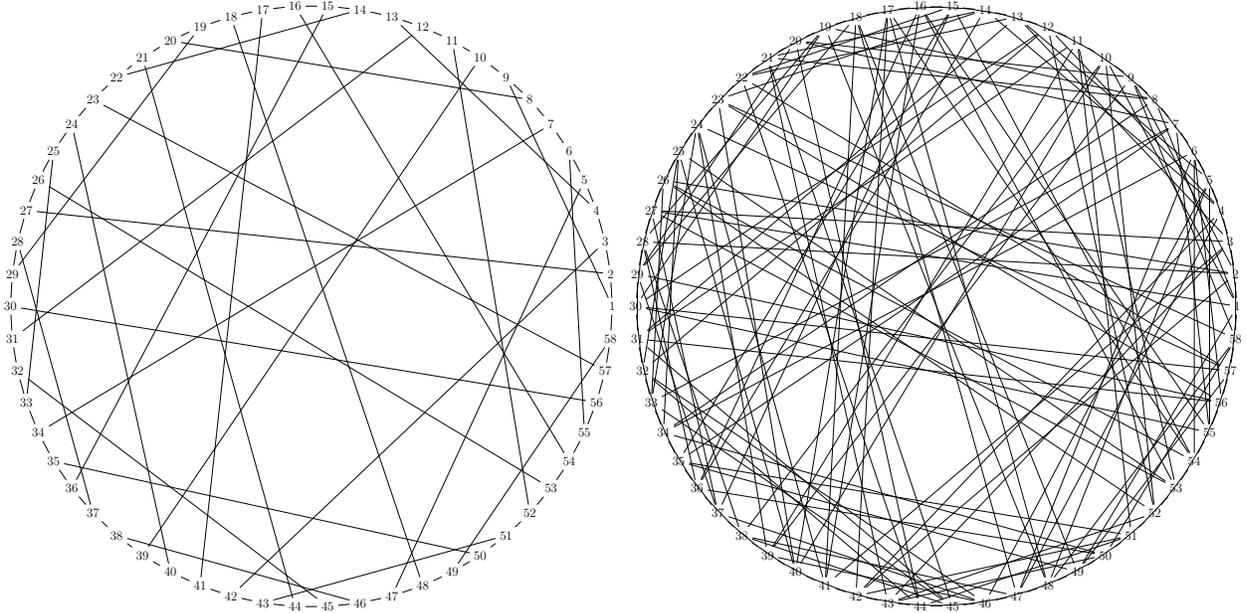

As a result of these challenges, we rely on computer assistance to upper bound the cop number of $G^2-E$. Given this, it is convenient for us to choose $G$ as small as possible, so that the problem remains computable. Fortunately, the graphs of smallest order with minimum degree $3$ and girth $9$ are known. These graphs are $(3,9)$-cages, the first of which was discovered in \cite{biggs1980trivalent}, with all $18$ being later characterised in \cite{brinkmann1995smallest}. We adopt the labeling $G_1,\dots, G_{18}$ of the $(3,9)$-cages given in \cite{brinkmann1995smallest}. All $(3,9)$-cages are order $58$, and for each $1\leq i \leq 18$ we let $H_i = G_i^2 - E(G_i)$. Note that for all $i\neq j$, $H_i$ and $H_j$ are not isomorphic. This was verified by computer and can be checked at \cite{Clow2024Git}.

Next, we prove that $17$ of the $18$ graphs $H_i$ have cop number $3$ and attacking cop number $6$. For an example of such a graph $H$ see Figure~\ref{fig: G_1,H_1}. Curiously, $\cop(H_{18})>3$, however we were unable to compute if $\cop(H_{18}) = 4$. We can verify that $\cop(H_{18})>3$ while not being able to verify if $\cop(H_{18}) = 4$ because the best known algorithm for checking if $\cop(G) \leq k$ is given by \cite{petr2022faster} and is $O(kn^{k+2})$ time, while the code we use from \cite{Afanassiev2017} implements an algorithm appearing \cite{bonato2011game} which is $O(n^{3k+3})$ time. The value of $\attCop(H_{18})-\cop(H_{18})$ remains open.

\begin{theorem}\label{Thm: cc = 2c =6}
    For all $1 \leq i \leq 17$, $\cop(H_i)=3$ and  $\attCop(H_i)=6$.
\end{theorem}

\begin{proof}

Consider $G_i$ for a fixed but arbitrary $1 \leq i \leq 17$. It was shown in \cite{biggs1980trivalent} that $G_i$ is a $(3,9)$-cage. That is, $G_i$ is $3$-regular and has girth $9$. Then, letting $H_i = G_i^2 - E(G_i)$, we note that 
$$
\gamma(H_i) \geq \frac{|V(H_i)|}{\Delta(H_i)+1} = \frac{58}{7} > 8
$$
so Lemma~\ref{Lemma: Square Lower Bound} implies that $\attCop(H_i) \geq 6$. 

For all $1 \leq i \leq 17$, $\cop(H_i)\leq3$ is verified using computer assistance. Our code is available at \cite{Clow2024Git}. We use the networkx package \cite{SciPyProceedings_11} to construct our graphs and code appearing at \cite{Afanassiev2017} to compute cop numbers. From here, $\attCop(H_i) \leq 2\cop(H_i)$ implies that $\cop(H_i) \geq \frac{\attCop(H_i)}{2} \geq \frac{6}{2} = 3$. Thus, $\cop(H_i)=3$ as required. Note that this implies $6 \leq \attCop(H_i) \leq 2\cop(H_i)=6$ so $\attCop(H_i)=6$ as required. This completes the proof.
\end{proof}

\section{Future Work}

Despite being introduced over a decade ago and being, at least in the authors' estimation, a natural variant, very little work has been done on Cops and Attacking Robbers. Given this, there remain many open questions regarding the attacking cop number. This section is devoted to introducing and motivating these questions.

To begin, can we extend the work from Section~3 to graphs which contain triangles? As demonstrated in Figure~\ref{fig: cc2 on triangle}, doing this will be non-trivial as Lemma~\ref{Lemma: The needed lower bound} is not true for graphs that contain triangles. We believe that solving this problem for all graphs is reasonable, and should be matter of future work.

\begin{problem}\label{Prob: Which graphs have cc = 2}
    Characterise the connected graphs with attacking cop number $2$.
\end{problem}

Recalling that characterising graphs with attacking cop number $1$ is easy, and every cop-win graph has attacking cop number $1$ or $2$, this problem reduces to characterising which graphs $G$ with cop number $2$ have $\cop(G) = \attCop(G)$. This begs the question, how important is the fact that $\cop(G) = 2 = \attCop(G)$ as opposed to $\cop(G) = k = \attCop(G)$ for some $k > 2$? More formally, observe the following question.

\begin{question}\label{Question: Qhat sufficient conditions for c = cc}
    What are some sufficient conditions for a graph $G$ to satisfy $\cop(G) = \attCop(G)$? Can we identify any necessary conditions for this to be true?
\end{question}

We next move our attention to planar graphs. While we have shown that all bipartite planar graphs have attacking cop number at most $4$, it remains unclear if there exists a planar graph with attacking cop number $5$ or even $6$. However, the only examples of graphs $H$ we know of with $\attCop(H) - \cop(H) > 1$ are far from being planar. We conjecture that this is no coincidence. Notice that if true, then Conjecture~\ref{Conj: Planar Strongest Upper Bound} would imply that all planar graphs have attacking cop number at most $4$. 

\begin{conjecture}\label{Conj: Planar Strongest Upper Bound}
    For all planar graphs $G$, $\attCop(G) \leq \cop(G) + 1$.
\end{conjecture}

Along these same lines, we recall Theorem~9 from \cite{bonato2014robber}, which states that if $G$ is bipartite then $\attCop(G) \leq \cop(G) + 2$. We notice that all examples of a graphs $H$ we know of with $\attCop(H) - \cop(H) > 1$ contain triangles and are therefore not bipartite. As with planar graphs we do not believe this is a coincidence. However, bipartite graphs form a much more complicated class to study in pursuit-evasion games than planar graphs. As a result we do not believe there is sufficient evidence for us to conjecture that for all bipartite graphs $G$, $\attCop(G) \leq \cop(G)+1$. Instead we pose the following question.

\begin{question}\label{Question: Bipartite cc = c + 2}
    Does there exist a bipartite graph $G$ with $\attCop(G) = \cop(G) + 2$?
\end{question}

We now proceed to our discussion of general graphs. In particular, we will continue to focus on how large the difference between cop number and attacking cop number can be. We begin with the following conjecture, which was stated in plain text earlier in the paper. 
\begin{conjecture}\label{Conj: Weak}
    For all non-negative integers $k$ there exists a graph $H$ such that $$\attCop(H)-\cop(H) \geq k.$$
\end{conjecture}

We note that there exists an even stronger conjecture we can make regarding this difference. That is, that for all $k$, the upper bound $\attCop(G) \leq 2\cop(G) = 2k$ from \cite{bonato2014robber} is tight. We can state this more formally as follows.

\begin{conjecture}\label{Conj: Strong}
    For all integers $k \geq 4$ there exists a graph $H$ such that $\cop(H) = k$ and 
    $$\attCop(H) = 2\cop(H).$$
\end{conjecture}

It is unclear if Conjecture~\ref{Conj: Strong} is true and gaining intuition on the matter has proven challenging. Perhaps the best evidence of how challenging the problems is, is that constructing graphs $H$ with $\attCop(H) - \cop(H) \geq 2$ has proven nontrivial. To the authors knowledge, only $18$ such graphs are known, $17$ of which first appear in this paper and one of which is given in \cite{bonato2014robber} (the line graph of the Peterson graph). Moreover all these graphs are obtained by a combination of a lower bound on attacking cop number using girth and minimum degree, see Lemma~8 from \cite{bonato2014robber} and Lemma~\ref{Lemma: Square Lower Bound}, and a direct computation of the cop number of a candidate graph. Such approaches will not generalise to large $k$ given computing the cop number of a graph is computationally hard for large $k$. As a result, new tools, such as upper bounds on $\cop(G^2-E)$, are required to make further progress on this problem. This seems to be an exciting direction of study, as for graphs $G$ of large girth $g$, the graph $G^2-E$ seems to exhibit many of the same properties as a graph of girth $g/2$ despite containing large cliques. This relationship seems of natural interest given how often considering graphs of girth at least $5$ is used to lower bound the cop number of a given class of graphs. 

It may also be worthwhile to see how far the similarity between $G^2-E$ and a graph of girth $g/2$ goes for its own sake. For example, does a result like the lower bound $\cop(G) \geq \frac{1}{g}(\delta-1)^{\lfloor \frac{g-1}{4}\rfloor}$ where $G$ has girth $g$ and minimum degree $\delta$ from \cite{bradshaw2023cop} hold for graphs $G^2-E$? Along these lines observe that if Conjecture~\ref{Conj: Strong} can be proven in the same way as Theorem~\ref{Thm: cc = 2c =6}, then this implies $\cop(G^2-E)$ would be much less than $\cop(G)$ for some graphs $G$. A result which seems nontrivial in its own right.

Next, we conjecture that the line graph of the Peterson graph, and the constructions we give in Theorem~\ref{Thm: cc = 2c =6} are minimal. 

\begin{conjecture}\label{Conj: Minimal Diff}
    Let $H$ is a graph with $\attCop(H)-\cop(H) \geq k$. If $k = 2$, then $H$ is order $n\geq 15$. If $k = 3$, then $H$ is order $n \geq 58$.
\end{conjecture}

Finally, we ask a more specific question which we were unable to answer. That is, what is the cop number and attacking cop number of $H_{18}$? In particular, what is their difference?

\begin{question}
    What is the value of $\attCop(H_{18}) - \cop(H_{18})$?
\end{question}

\section*{Acknowledgements}
M.E. Messinger acknowledges research support from NSERC (grant application 2018-04059). Melissa A. Huggan acknowledges research support from NSERC (grant application 2023-03395).

\bibliographystyle{plain}
\bibliography{Cops&AttRobbers}

\end{document}